\documentclass{amsart}

\usepackage{amsmath}
\usepackage{amsthm}
\usepackage{amsfonts}
\usepackage{hyperref}
\usepackage{enumerate}
\usepackage[all]{hypcap}[2006/02/20]
\usepackage{color}

\newcommand\R{\mathbb{R}}
\newcommand\rd{\mathrm d}

\newcommand{\bb}[1]{\mathbb{#1}}

\newcommand{\pard}[2]{\frac{\partial #1}{\partial #2}}

\newcommand{\ho}{\left(\frac{\rd}{\rd t} -\Delta \right)}
\newcommand{\ddt}[1]{\frac{ \rd #1}{\rd t}}

\renewcommand{\b}{\beta}

\newcommand{\n}{\nabla}
\newcommand{\e}{\epsilon}
\newcommand{\s}{\sigma}
\renewcommand{\S}{\Sigma}
\renewcommand{\o}{\omega}
\renewcommand{\t}{\theta}
\newcommand{\T}{\Theta}
\newcommand{\ot}{\ov{\theta}}
\renewcommand{\r}{\rho}

\renewcommand{\l}{\lambda}
\renewcommand{\d}{\delta}
\renewcommand{\a}{\alpha}
\newcommand{\Ha}{\mathcal{H}}
\newcommand{\Si}{\Sigma}
\newcommand{\ra}{\rightarrow}
\newcommand{\ov}{\overline}

\newcommand{\C}{\mathbb{C}}
\newcommand{\Lo}{{\L}ojasiewicz}
\newcommand{\Cu}{\mathcal{C}}
\DeclareMathOperator\vol{vol}
\DeclareMathOperator\Ree{Re}
\DeclareMathOperator\Imm{Im}
\DeclareMathOperator\Sing{Sing}

\numberwithin{equation}{section}
\newtheorem{theorem}{Theorem}[section]

\newtheorem{lem}[theorem]{Lemma}

\newtheorem{prop}[theorem]{Proposition}
\newtheorem{cor}[theorem]{Corollary}
\theoremstyle{definition}
\newtheorem{defn}[theorem]{Definition}	

\newtheorem{ex}[theorem]{Example}
\theoremstyle{remark}
\newtheorem{remark}[theorem]{Remark}

\begin{document}

\title[Ancient solutions in LMCF]{Ancient solutions in Lagrangian mean curvature flow}

\author{Ben Lambert}
\author{Jason D. Lotay}
\author{Felix Schulze}
\address{Department of Mathematics, University College London, Gower Street, London, WC1E 6BT, United Kingdom}
\email{b.lambert@ucl.ac.uk; f.schulze@ucl.ac.uk}
\address{Mathematical Institute, University of Oxford, Woodstock Road, London, OX2 6GG, United Kingdom}
\email{jason.lotay@maths.ox.ac.uk}


\begin{abstract} 
Ancient solutions of Lagrangian mean curvature flow in $\bb{C}^n$ naturally arise as Type II blow-ups.  In this extended note we give structural and classification results for such ancient solutions in terms of their blow-down and, motivated by the Thomas--Yau Conjecture, focus on the almost calibrated case.  In particular, we classify Type II blow-ups of almost calibrated Lagrangian mean curvature flow when the blow-down is a pair of transverse planes or, when $n=2$, a multiplicity two plane.  We also show that the Harvey--Lawson Clifford torus cone in $\bb{C}^3$ cannot arise as the blow-down of an almost calibrated Type II blow-up.
\end{abstract}

\maketitle

\section{Introduction}

Ancient solutions are known to occur naturally in the theory of mean curvature flow as Type II blow-ups of singularities of a flow, see for example \cite{HuiskenSinestrariTypeII}. In this article we study ancient solutions to Lagrangian mean curvature flow (LMCF for short) in $\bb{C}^n$, that is, solutions $(L_t)_{t\in I}$ which exist for all $t \in I = (-\infty,0)$. Neves \cite{Neves.ZM} has shown that in the zero Maslov class case, no singularities of Type I can form and that singularities are, in a sense, unavoidable without further assumptions 
\cite{Neves.Singularities}. Even more, Neves shows that in this case any tangent flow consists of a finite union of special Lagrangian cones. Thus, to better understand singularity formation, a first step is to classify ancient solutions which arise as a Type II blow-up.

The Thomas--Yau Conjecture \cite{ThomasYau} asserts that the long-time existence and convergence of solutions to LMCF which are \emph{almost calibrated} (see Definition \ref{def:almost-calibrated}) is equivalent to a ``stability condition'', as is the case for Hermitian--Yang--Mills flow.  Motivated by this conjecture, we focus on almost calibrated ancient solutions to LMCF.  It is worth observing that the almost calibrated condition rules out the singularities constructed in \cite{Neves.Singularities}.     The stability condition in \cite{ThomasYau} is given in terms of the Lagrangian angle; see \cite{JoyConj} instead for conjectures relating long-time existence and convergence of LMCF to Bridgeland stability conditions and Fukaya categories.

As a consequence of Huisken's monotonicity formula, a blow-down of an ancient solution to mean curvature flow with uniformly bounded area ratios subconverges to a self-similarly shrinking (weak) mean curvature flow. In this paper we study the blow-down of ancient solutions to LMCF with the aim of using this to characterise the solution.  Our first main result is as follows.

\begin{theorem}\label{thm:Lawlor}
 Let $(L_t)_{-\infty<t<0}$ be an ancient exact almost calibrated solution to LMCF in $\bb{C}^n$ with uniformly bounded area ratios and uniformly bounded Lagrangian angle, see \eqref{eq:bounds}. Suppose a blow-down $(L^\infty_s)_{-\infty<s<0}$ of $(L_t)_{-\infty<t<0}$ is a pair of multiplicity one transverse planes $P_1\cup P_2$. Then either
 \begin{itemize}
  \item[(a)] $L_t = P_1\cup P_2$ for all $t<0$, or
  \item[(b)] up to rigid motions, $L_t$ is a Lawlor neck for all $t<0$. 
  \end{itemize}
\end{theorem}

\noindent For the definition of a solution being exact 
see Definition \ref{def:exact}. 
For the definition of a Lawlor neck, see Example \ref{ex:Lawlor}.  

The next result rules out a blow-down which is the simplest known non-planar special Lagrangian cone.

\begin{theorem}\label{thm:HL}
Let $(L_t)_{0\leq t<T}$ be an almost calibrated solution to LMCF with uniformly bounded area ratios in a Calabi--Yau 3-fold, with a singularity at $T$.  The blow-down of a Type II blow-up cannot be the Harvey--Lawson $T^2$-cone $C$, see \eqref{eq:HL.cone}.  
\end{theorem}

If $n=2$ we may use the algebraic structure of special Lagrangians to yield an improved result, allowing also some higher multiplicity of the blow-down. 

\begin{theorem}\label{thm:dimension2}
Suppose that $(L_t)_{-\infty<t<0}$ is an ancient almost calibrated solution to LMCF in $\C^2$ satisfying \eqref{eq:bounds} whose blow-down is a plane $P$ with multiplicity two (with a single Lagrangian angle).  Then either
\begin{itemize}
\item[\emph{(a)}] $L_t=P$ for all $t<0$, or
\item[\emph{(b)}] up to rigid motions, $L_t$ is a special Lagrangian given by Example \ref{ex:SL.z2} for all $t<0$.
\end{itemize}
\end{theorem}

\noindent Note that this yields a unique model, describing how, in the almost calibrated case, LMCF can develop double density planes as tangent flows.

The proof of these results requires several steps, which are of independent interest. We describe these steps in the following overview of the paper.

We recall some basic definitions in Section \ref{sect:Primliminaries}. In Section \ref{sect:blowdowns}, we provide the proofs of a Neves-type structure theory  for blow-downs of ancient solutions, which follows with only minor modifications from \cite{Neves.ZM, NevesTian}. In Section \ref{sect:stationary} we apply these results to almost calibrated ancient solutions to show that if a blow-down is a pair of planes with different angles, then $L_t$ is a stationary pair of planes for all $t<0$ (cf.~\cite[Corollary 4.3]{Neves.Survey}). We additionally show that if the blow-down of the ancient solution is special Lagrangian, then the ancient solution is already a stationary special Lagrangian, see also \cite{NevesTian}. The first theorem follows after applying the rigidity results of Imagi, Joyce and Oliviera dos Santos \cite{IJO.Uniqueness.Lawlor} for Lawlor necks. However, to apply these results we have to weaken the asymptotics assumed in \cite{IJO.Uniqueness.Lawlor}. For this step we have to improve sub-linear decay to the asymptotic cone to decay like $O(r^{\alpha})$, where $r$ is the distance from the origin and $\alpha \in (0,1)$. This is achieved by an optimal \Lo--Simon inequality following an argument similar to Simon's proof of the uniqueness of tangent cones with isolated singularities, see \cite[Section 7]{SimonCone}. This is done in Section \ref{sect:asymptotics}. 

As our main source of ancient solutions is Type II blow-ups, we also include a result based on an idea by Cooper \cite{Cooper} which restricts the topology of such blow-ups: Proposition \ref{prop:top.blowup} states that any Type II blow-up must be exact and zero-Maslov. The work of Imagi \cite{Imagi} shows that any special Lagrangian smoothing of the Harvey-Lawson $T^2$ cone is non-exact, which then implies Theorem \ref{thm:HL}.

The third theorem follows from the observation that special Lagrangians are hyperk\"ahler rotations of complex curves which may be written as the zero sets of polynomials, the degree of which is determined by the number of planes in the blow-down counted with multiplicity. The details of these results are contained in Section \ref{sectionSLsurfaces}.  

\textbf{Acknowledgements.}
The authors were supported by a Leverhulme Trust Research Project Grant RPG-2016-174.
\section{Preliminaries}\label{sect:Primliminaries}

We start with some basic material that will be useful throughout the article.

\subsection{Lagrangians and special Lagrangians} Let $\C^n$ be endowed with its standard complex coordinates $z_j=x_j+iy_j$, for $j=1,\ldots,n$,   complex structure $J$,  
K\"ahler form $\omega=\sum_{j=1}^n\rd x_j\wedge\rd y_j$ and holomorphic volume form $\Omega=\rd z_1\wedge\ldots\wedge\rd z_n$.  Notice that the Liouville form 
\begin{equation}\label{eq:Liouville}
\lambda=\sum_{j=1}^n(x_j\rd y_j-y_j\rd x_j)
\end{equation}
satisfies $\rd\lambda=2\omega$.

Let $L$ be a \emph{Lagrangian} in $\C^n$; that is, a (real) $n$-dimensional (smooth) submanifold of $\C^n$ such that $\omega|_L\equiv 0$.  This condition is vacuous for $n=1$ (i.e.~all curves are automatically Lagrangian), so we will always assume that $n\geq 2$.   
Since $\lambda|_L$ is clearly closed, we have a special class of Lagrangians  for which $\lambda|_L$ is exact.  

\begin{defn}\label{def:exact}
A Lagrangian $L$ in $\C^n$ is \emph{exact} if the restriction of the Liouville form \eqref{eq:Liouville}, i.e.~$\lambda|_L$, is exact.
\end{defn}

As $\Omega|_L$ is unit length, if $L$ is oriented then there exists a \emph{phase} function $e^{i\t }:L\to \mathcal{S}^1$ such that 
\begin{equation}\label{eq:Lag.angle}
\Omega|_L=e^{i\t }\vol_L,
\end{equation} where $\vol_L$ is the induced volume form on $L$.  We can generalise this discussion from $\C^n$ to whenever the ambient manifold is a \emph{Calabi--Yau $n$-fold}, since one still has a K\"ahler form $\omega$ and a nowhere vanishing holomorphic volume form $\Omega$.

\begin{defn}
For a  Lagrangian $L$ in a Calabi--Yau $n$-fold with \emph{phase} $e^{i\t}$ defined by \eqref{eq:Lag.angle}, we call the (possibly multi-valued) function $\t$ the \emph{Lagrangian angle} of $L$.    When $\t $ is single-valued we say that $L$ is \emph{zero-Maslov}: this is equivalent to saying that the Maslov class in $H^1(L)$, which is proportional to $[\rd\theta]$, vanishes.
\end{defn}

It is important to note that the mean curvature vector $H$ and position vector $\mathbf{x}$ of an oriented Lagrangian $L$ in $\C^n$ with Lagrangian angle $\theta$ satisfy
 \begin{equation}\label{eq:magic.formulae}
 H\lrcorner\omega|_L=-\rd\theta\quad\text{and}\quad \mathbf{x^{\perp}}\lrcorner\omega|_L=\lambda|_L,
 \end{equation}
 where $\{\,\}^{\perp}$ is orthogonal projection onto the normal bundle of $L$.  We observe the following consequences of \eqref{eq:magic.formulae}, noting the first equation holds in any Calabi--Yau.
 
\begin{lem}\label{lem:magic.facts} 
\begin{itemize}
\item[]
\item[(a)] A Lagrangian in a Calabi--Yau $n$-fold is minimal if and only if its Lagrangian angle is constant on connected components.
\item[(b)] A Lagrangian $L$ in $\C^n$ is a cone if and only if $\lambda|_L=0$.
\end{itemize}
\end{lem}

\begin{remark}\label{rmk:components.cone} In (a), we view a Lagrangian $L$ as the image of an immersion $\iota$, and abuse notation by saying that a connected component of $L$ is a connected component of its pre-image under $\iota$.  We will see the utility of this presently.  In (b), by a cone we simply mean a set invariant under the action of dilations.
\end{remark}

This leads us to an obvious special class of zero-Maslov Lagrangians: those for which the Lagrangian angle is constant.

\begin{defn}
A Lagrangian $L$ in a Calabi--Yau $n$-fold is \emph{special Lagrangian} if the Lagrangian angle $\t $ is constant.  Equivalently, there is a constant $\ov\t$ such that
\begin{equation}\label{eq:SL}
\Imm(e^{-i\ov\t}\Omega|_L)=0,
\end{equation} which means (up to a choice of orientation) that $\t=\ov\t$.
\end{defn}

\begin{remark}
Special Lagrangians are calibrated (by $\Ree(e^{-i\ov\t}\Omega)$ for some $\ov\t$) and so are area-minimizing (rather than just minimal).  Moreover, this description allows us to define  
special Lagrangian integral currents, with a given Lagrangian angle.
\end{remark}

\begin{remark}\label{rmk:planes.angles} Any oriented Lagrangian $n$-plane through $0$ in $\C^n$ is special Lagrangian, but the union of two such planes with different phases    need not be special Lagrangian.  In particular, it is well-known that the angle between a transversely intersecting pair of planes has to be sufficiently large for the union to be area-minimizing (see e.g.~\cite{Lotay.CG}).  
In fact, in this case the difference between the angles must be a multiple of $\pi$ for the union to be special Lagrangian with the same phase.
\end{remark}

\noindent Remark \ref{rmk:planes.angles} shows that even though a pair of intersecting planes is connected, the Lagrangian angle can take different values on each plane, which justifies the conventions in Remark \ref{rmk:components.cone}. 

It is important to notice that 2-dimensional special Lagrangians have an exceptional character, since Calabi--Yau 2-folds are hyperk\"ahler (see, e.g.~\cite{Lotay.CG}).

\begin{lem}\label{lem:SL.cplx}
A Lagrangian $L$ in a Calabi--Yau $2$-fold $M$ is special Lagrangian if and only if there is a hyperk\"ahler rotation of $M$ such that $L$ is a complex curve. 
\end{lem}

There is another important class of Lagrangians we shall study, defined as follows. 
\begin{defn}\label{def:almost-calibrated}
A Lagrangian $L$ in a Calabi--Yau $n$-fold with Lagrangian angle $\theta$ is \emph{almost calibrated} if there is a constant $\epsilon>0$ so that 
\begin{equation}\label{eq:almost.calibrated}
\cos\theta\geq\epsilon.
\end{equation} Notice that an almost calibrated Lagrangian is clearly zero Maslov. \end{defn}

\noindent The Lagrangian angle of almost calibrated Lagrangians can be chosen to lie in  $(-\frac{\pi}{2},\frac{\pi}{2})$, so they are potential candidates to be deformed to special Lagrangians with Lagrangian angle $0$.

Finally, in $\C^n$, we have an important class of Lagrangians in symplectic topology, which will also play a role in the Lagrangian mean curvature flow.

\begin{defn}
A Lagrangian $L$ in $\C^n$ is \emph{monotone} if there exists a constant $m>0$ such that 
\begin{equation*}
[\rd\theta]=m[\lambda|_L]\in H^1(L),
\end{equation*}
where $\theta$ is the Lagrangian angle of $L$.  
\end{defn}

\subsection{Lagrangian mean curvature flow}

Mean curvature flow in a Calabi--Yau $n$-fold (and thus in $\C^n$) preserves the Lagrangian condition, thus leading to the notion of Lagrangian mean curvature flow (LMCF).  The critical points of this flow are exactly (unions of) special Lagrangians and the flow also preserves the classes of exact, zero-Maslov and almost calibrated Lagrangians.  Therefore, given a solution $(L_t)_{t\in I}$ of LMCF we will let $\lambda_t=\lambda|_{L_t}$ and let $\t_t$ be the Lagrangian angle of $L_t$.

Two fundamental concepts in the study of mean curvature flow are that of blow-up and blow-down. We define two types of blow-ups.

\begin{defn}\label{defn:type.I}
Let $(L_t)_{0\leq t<T}$ be a solution to LMCF in $\C^n$ and let $x_0\in\C^n$.  For a positive sequence $\sigma_i\to\infty$, we define the (\emph{centred} or \emph{Type I}) \emph{blow-up} sequence at $(x_0,T)$ by
\begin{equation*}
L^i_s=\sigma_i(L_{T+\sigma_i^{-2}s}-x_0)\quad \forall\, s\in [-\sigma_i^2T,0).
\end{equation*}   
The sequence $(L^i_s)_{-\sigma_i^2T \leq s<0}$ always subconverges weakly (i.e.~as a Brakke flow) as $i\ra\infty$, by \cite[Lemma 8]{Ilmanen.Singularities}, to a limit flow $L^{\infty}_s$ in $\C^n$ for all $s<0$, which is called a \emph{tangent flow} at $(x_0,T)$.  A solution to LMCF which is defined for all times in $(-\infty,T_0)$ for some $T_0$ is called an \emph{ancient solution}: we will always take $T_0=0$.
\end{defn}

\noindent More generally, let $(L_t)_{0\leq t<T}$ be a solution to LMCF in a Calabi--Yau $n$-fold $M$ and let $x_0\in M$.  If $g$ is the Calabi--Yau metric on $M$, then for a positive sequence $\sigma_i\to\infty$ we consider flows $L^i_s=L_{T+\sigma^{-2}_is}$ in the manifolds $(M,\sigma_i^2g)$.  The sequence of pointed manifolds $(M,\sigma_i^2g,x_0)$ converges to $\C^n\cong T_{x_0}M$ with the flat metric as $i\ra\infty$.  We then define the tangent flow  at $(x_0,T)$ by taking the limit of the
sequence $(L^i_s)_{-\sigma_i^2T \leq s<0}$ in $(M,\sigma_i^2g,x_0)$: again, the tangent flow will always exist as a Brakke flow and will give an ancient solution $(L^{\infty}_s)_{-\infty<s<0}$ in $\C^n$.

\begin{defn}\label{defn:type.II}
Let $(L_t)_{0\leq t<T}$ be a solution to LMCF in $\C^n$.  Suppose that we have a sequence $(x_i,t_i)\in \C^n\times (0,T)$ such that $t_i\to T$ and, if $A_t$ is the second fundamental form of $L_t$ we have
\begin{equation}\label{eq:secondfunform.blow.up}
\sigma_i:=A(x_i,t_i)=\sup\{|A_t(x)|\,:\,x\in L_t,\,t\leq t_i\}>0.
\end{equation} 
We then define the \emph{Type II blow-up} sequence by
\begin{equation*}
L^i_s=\sigma_i(L_{t_i+\sigma_i^{-2}s}-x_i) \quad\forall\, s\in[-\sigma_i^2t_i,0).
\end{equation*}
Since $|A^i_s|$ is now bounded by $1$ for all $i$, the sequence $L^i_s$ subconverges as $i\ra\infty$ and it will define a (smooth) ancient solution $L^{\infty}_s$ in $\C^n$, which we call the  \emph{Type II blow-up}.
\end{defn}

\noindent In a Calabi--Yau $n$-fold we adopt the same procedure, now using pointed convergence of the sequence of manifolds $(M,\sigma_i^2g,x_i)$.

\begin{remark} (i)
Definitions \ref{defn:type.I} and \ref{defn:type.II} are related to a well-known dichotomy when studying singularities in mean curvature flow.  At a singularity at time $T$ we know that $\sup_{L_t}|A_t(x)|\to\infty$ as $t\nearrow T$.  We say that the singularity is Type I if
\begin{equation}\label{eq:Type.I.blowup}
\lim_{t\nearrow T} \sup_{L_t}|A_t|^2(T-t)<\infty
\end{equation}
and we say it is Type II otherwise (see e.g.~\cite{HuiskenSinestrariTypeII}).\\
(ii) To construct a Type II blow-up and all the consequences thereof it is sufficient that the points $x_i$ stay in a bounded set and that there exists a monotone sequence of radii $R_i \rightarrow +\infty$ and a constant $C_0$ such that
$$ \sup_{B_{R_i}(0) \times (-R_i^2,0]} |A_s^i| \leq C_0 |A_s^i(0,0)|\, .$$
This is often helpful if one only wants to take a local supremum of the curvature.
\end{remark}

We now define the blow-down of an ancient solution to LMCF.

\begin{defn}\label{defn:Blowup}
Let $(L_t)_{-\infty<t<0}$ be an ancient solution to LMCF in $\C^n$ and we assume that it has uniformly bounded area ratios, i.e.~ there exists $C>0$ such that
\begin{equation}\label{eq:bounded.area.ratios}
\mathcal{H}^n(L_t\cap B_r(x)) \leq C r^n
\end{equation}
for all $t\in (\infty, 0)$, $x \in \C^n$, $r>0$, where $\mathcal{H}^n$ is the \emph{real} $n$-dimensional Hausdorff measure on $\R^{2n} \cong \C^n$.  For a positive  sequence $\sigma_i\to\infty$, we define the blow-down sequence of $(L_t)_{-\infty<t<0}$ to be 
\begin{equation}\label{eq:blow-down.seq}
L_s^i :=\s_i^{-1} L_{\s_i^2s} \qquad  \forall s<0,
\end{equation}
and the corresponding \emph{blow-down} $(L^{\infty}_s)_{-\infty<s<0}$ of $(L_t)_{-\infty<t<0}$ is a subsequential limit of the sequence $(L^i_s)_{-\infty<s<0}$, which is again an ancient solution to LMCF.  
\end{defn}

A key tool in studying mean curvature flow, and thus LMCF, is the following.

\begin{defn}\label{defn:backwards.heat.kernel}  
Given $(x_0,l)$ in $\C^{n}\times \R$, we consider the backwards heat kernel
\begin{equation}\label{eq:backward.heat}
\Phi_{(x_0,l)}(x,t)=\frac{\exp\left(-\frac{|x-x_0|^2}{4(l-t)}\right)}{(4\pi(l-t))^{n/2}}\,.
\end{equation}
We will also sometimes write $\Phi(x,t)=\Phi_{(0,0)}(x,t)$.  

Given a Lagrangian $L$ in $\C^n$, $x_0\in\C^n$ and $l>0$, we have Gaussian density ratios \begin{equation}\label{eq:gaussian}
\Theta(x_0,l)=\int_{L}\Phi_{(x_0,l)}{(x,0)} \, \rd\Ha^n.
\end{equation} 
Given a solution $(L_t)_{t\in I}$ to LMCF, we will let $\Theta_t$ be the Gaussian density ratios of $L_t$.
\end{defn}

Using \eqref{eq:magic.formulae}, Huisken's monotonicity formula \cite{Huisken.Monotonicity} implies that for a solution $L_t$ of LMCF and functions $f_t:L_t\to\R$:,
\begin{equation}\label{eq:monotonicity}
\ddt{}\int_{L_t} f_t \Phi_{(x_0,l)} \, \rd\Ha^n = \int_{L_t} \left[\ho f_t -\left|\rd\theta_t+\frac{\lambda_t}{2t}\right|^2f_t\right] \Phi_{(x_0,l)} \, \rd\Ha^n, 
\end{equation}
for $t<l$ and when these integrals are well-defined (e.g.~in the non-compact setting, when $L_t$ have uniformly bounded area ratios and $f_t$ have at most polynomial growth at infinity).

 The monotonicity formula \eqref{eq:monotonicity}, and other considerations, motivates the study of the following ancient solutions which are solitons or self-similar solutions.
 
 \begin{defn}
 An ancient solution $(L_t)_{-\infty<t<0}$ to LMCF in $\C^n$ is a \emph{self-shrinker} if
 \begin{equation*}
 \rd\theta_t+\frac{\lambda_t}{2t}=0 \quad \forall\, t<0.
 \end{equation*}
A self-shrinker is an ancient solution and simply moves by dilations under the flow, shrinking to a cone at $t=0$.  

A solution $(L_t)_{t\in \R}$ to LMCF in $\C^n$ is a \emph{translator} if there is a constant $1$-form $\alpha$ on $\C^n$ so that
\begin{equation*}
\rd\theta_t=\alpha|_{L_t}\quad\forall\,t\in\R.
\end{equation*}
Here, as $L_t$ is defined for all $t\in\R$ it is an \emph{eternal} solution, and moves by translations in the direction given by $J\alpha^\sharp$.
 \end{defn}
 
 \begin{remark}
 If we assume that $L_t$ has bounded area ratios, then any tangent flow and any centred blow-down will be self-shrinkers (in a weak sense) by 
  \eqref{eq:monotonicity}.  
 \end{remark}

 We make an  observation, which follows immediately 
 from the definitions.
 
 \begin{lem}\label{lem:self.monotone}
Lagrangian self-shrinkers are monotone for each $t$.  In particular, $L_t$ is zero-Maslov if and only if it is exact.
 \end{lem}
 
 Another consequence of \eqref{eq:monotonicity}  is that along LMCF  we have
 \begin{equation}\label{eq:Theta.dec}
\Theta_s(x_0,l)\leq \Theta_t(x_0,l+s-t)\quad\mbox{for all }x_0\in\C^n,\,s\geq t,\, l>0.
\end{equation}

The quantities  $\rd\theta_t$ and $\lambda_t$ satisfy evolution equations as follows (see e.g.~\cite{Neves.Survey}).

\begin{lem}\label{lem:evol.eqs}
Let $(L_t)_{t\in I}$ be a solution to LMCF in $\C^n$. Then
\begin{equation*}
\ho\rd\theta_t=0\quad\text{and}\quad \ho\lambda_t+2\rd\theta_t=0.
\end{equation*}
Hence, 
\begin{equation*}
\ddt{}[\rd\theta_t]=0\quad\text{and}\quad \ddt{}[\lambda_t]+2[\rd\t_t]=0.
\end{equation*}
\end{lem}

\noindent The evolution equation for $\rd\theta_t$ holds in any Calabi--Yau $n$-fold.

A crucial property of zero-Maslov LMCF then follows from Lemma \ref{lem:evol.eqs}.

\begin{lem} If $(L_t)_{t\in I}$ is a zero-Maslov solution to LMCF in a Calabi--Yau $n$-fold, then the Lagrangian angles $\t_t$ can be chosen to satisfy
\begin{equation}\label{eq:theta.evol}
\ho\t_t=0.
\end{equation}
\end{lem}

\noindent This evolution equation immediately implies that the almost calibrated condition \eqref{eq:almost.calibrated} is preserved under LMCF.  Moreover, given any $y\in\R$ and $m\in\bb{Z}$, we have 
\begin{equation}\label{eq:theta.evol.2}
\ho(\t_t-y)^{2m} + 2m(2m-1)(\t_t-y)^{2m-2}|\rd\t_t|^2=0.
\end{equation}

If $(L_t)_{t\in I}$ is a solution to LMCF in $\C^n$ consisting of exact Lagrangians, then we may write $\lambda_t=\rd\beta_t$ for functions $\beta_t:L\to\R$.  If, furthermore, the $L_t$ are zero-Maslov, then the $\beta_t$ satisfy a good evolution equation, again by Lemma \ref{lem:evol.eqs}

\begin{lem} If $(L_t)_{t\in I}$ is an exact zero-Maslov solution to LMCF in $\C^n$, then the functions $\beta_t$ may be chosen so that
\begin{equation}\label{eq:beta.evol}
\ho\beta_t+2\t_t=0.
\end{equation}
\end{lem}

Combining \eqref{eq:theta.evol} and \eqref{eq:beta.evol} shows that
\begin{equation}\label{eq:beta.evol.2}
\ho(\beta_t+2t\theta_t)=0.
\end{equation}

\subsection{Examples}  We now look at some important examples of ancient solutions, specifically some special Lagrangians, self-shrinkers and translators in $\C^n$.

\subsubsection{Lawlor necks}  

The simplest possible singularity model is a transversely intersecting pair of planes, so it is important to know how to resolve such a singularity.  In $\C^2$, one sees easily which pairs of planes are special Lagrangian and their special Lagrangian smoothings via Lemma \ref{lem:SL.cplx}, since we may appeal to complex geometry.

\begin{ex}\label{ex.SL.dim2}
Up to changes of complex coordinates, there is only one pair of transverse complex planes in $\C^2$, namely that defined by 
\begin{equation}
\{(w_1,w_2)\in\C^2\,:\,w_1w_2=0\}.
\end{equation}
This has a unique family of complex smoothings, given by
\begin{equation}
\{(w_1,w_2)\in\C^2\,:\,w_1w_2=c\},
\end{equation}
for $c\in\C\setminus\{0\}$.  Performing a hyperk\"ahler rotation, if we let $c=a+ib$ for $(a,b)\in\R^2\setminus\{(0,0)\}$, we find special Lagrangians (with Lagrangian angle $0$)
\begin{equation}
L_{\frac{\pi}{2}}(a+ib)=\left\{\left(x_1+i\frac{ax_1+bx_2}{x_1^2+x_2^2},x_2+i\frac{ax_2-bx_1}{x_1^2+x_2^2}\right)\,:\,(x_1,x_2)\in\R^2\setminus\{(0,0)\}\right\},
\end{equation}
which are embedded, diffeomorphic to $\mathcal{S}^1\times\R$ and asymptotic to  
 $P_0\cup P_{\frac{\pi}{2}}$, where
\begin{equation}\label{eq:P.phi}
P_{\phi}:=\{(e^{i\phi}x_1,e^{-i\phi}x_2)\,:\,(x_1,x_2)\in\R^2\}.
\end{equation}
It is elementary to see that

 $L_{\frac{\pi}{2}}(a+ib)$ is exact if and only if $b=0$. The Lagrangian $L_{\frac{\pi}{2}}(a)$ for $a>0$ is $\textrm{SO}(2)$-invariant and is called the \emph{Lagrangian catenoid}.

By varying complex coordinates we see that the pairs of transverse special Lagrangian planes are given by $P_0\cup P_{\phi}$ for $\phi\in(0,\pi)$.  One can follow the construction above to find a 2-parameter family of special Lagrangians $L_{\phi}(a+ib)$ with Lagrangian angle $0$ asymptotic to $P_0\cup P_{\phi}$ so that $L_{\phi}(a+ib)$ is exact if and only if $b=0$. 

\end{ex}

\noindent By the theory of resolution of singularities for complex curves, the $L_{\phi}(a+ib)$ are (up to rigid motions) the unique special Lagrangians in $\C^2$ which converge weakly to the pair of planes $P_0\cup P_{\phi}$, which are not the pair of planes itself. 

Further examples of special Lagrangians in $\C^2$ will appear in our study.

\begin{ex}\label{ex:SL.z2}
For $c=a+ib\in\C\setminus\{0\}$, consider the complex curve
\begin{equation}\label{eq:z2}
\{(cw,w^2)\in\C^2\,:\,w\in\C\}.
\end{equation}
After a hyperk\"ahler rotation we get special Lagrangians (with Lagrangian angle $0$)
\begin{equation}\label{eq:SL.z2}
L(a+ib)=\big\{\big(ax_1-bx_2+i(x_1^2-x_2^2),ax_2+bx_1-2ix_1x_2\big)\,:\,(x_1,x_2)\in\R^2\big\},
\end{equation}
which are embedded, diffeomorphic to $\R^2$, exact and asymptotic to $P_{\frac{\pi}{2}}$ as given in \eqref{eq:P.phi} with multiplicity two.   
\end{ex}

The exact special Lagrangians in Example \ref{ex.SL.dim2} can be extended to higher dimensions.

\begin{ex}\label{ex:Lawlor}  
Given $\phi_j\in (0,\pi)$ for $j=1,\ldots,n-1$ (for $n\geq 2$)  such that
\begin{equation}
\phi_n:=-(\phi_1+\ldots+\phi_{n-1})\in(-\pi,0),
\end{equation}
we let $\mathbf{\phi}=(\phi_1,\ldots,\phi_{n-1})$ and define a special Lagrangian plane (with Lagrangian angle $0$) in $\C^n$ by
\begin{equation}\label{eq:plane}
P_{\mathbf{\phi}}=\{(e^{i\phi_1}x_1,\ldots,e^{i\phi_n}x_n)\,:\,(x_1,\ldots,x_n)\in\R^n\}.
\end{equation}
There is a 1-parameter family of exact, embedded special Lagrangians $L_{\mathbf{\phi}}(a)$ in $\C^n$ for $a\in\R\setminus\{0\}$, which are diffeomorphic to $\mathcal{S}^{n-1}\times\R$ and asymptotic to $P_0\cup P_{\mathbf{\phi}}$ where $P_0=\R^n$.  

The $L_{\mathbf{\phi}}(a)$ are called \emph{Lawlor necks} and the case where $\phi_j=\frac{\pi}{2}$ for all $j$ gives an $\mathrm{SO}(n)$-invariant special Lagrangian, also often called the Lagrangian catenoid in $\C^n$.  For $n=2$, they coincide with the $L_{\phi}(a)$ given in Example \ref{ex.SL.dim2}.
\end{ex}

\noindent By \cite{IJO.Uniqueness.Lawlor}, up to rigid motions, Lawlor necks in $\C^n$ for $n>2$ are the \emph{unique} embedded, exact, special Lagrangians asymptotic to $P_0\cup P_{\mathbf{\phi}}$ in the sense that, outside some compact set they may be written as the graph of a normal vector field $v$ on the pair of planes so that, for some $\rho<1$, 
\begin{equation}\label{eq:asymptotics}
|v|=O(r^{\rho})\quad\text{and}\quad |\nabla v|=O(r^{\rho-1})\quad\text{as }r\ra\infty,
\end{equation}
where $r$ is the radial coordinate in the planes.  Furthermore, if we allow for immersed, exact, special Lagrangians the only additional possibility is the pair of planes itself.  For $n=2$, the work of \cite{IJO.Uniqueness.Lawlor} does not apply, but (as we have remarked) a stronger uniqueness holds which does not require \eqref{eq:asymptotics}.

\subsubsection{Grim reaper and Joyce--Lee--Tsui translators}

We recall that any curve $\gamma$ in $\C$ is Lagrangian, 
and therefore $\gamma\times\R$ is Lagrangian in $\C\times \C=\C^2$.  %
\begin{ex}\label{ex:grim}
The \emph{grim reaper} curve 
$$\gamma=\{s-i\log \cos s\in\C\,:\,s\in(-\textstyle\frac{\pi}{2},\textstyle\frac{\pi}{2})\}$$
has Lagrangian angle taking all values in the range $(-\frac{\pi}{2},\frac{\pi}{2})$, so this is zero-Maslov but is \emph{not} almost calibrated.  This defines a translator whose blow-down is a multiplicity two line.
Therefore $\gamma\times\R$ is a zero-Maslov but not almost calibrated translator in $\C^2$ whose blow-down is a multiplicity two plane.   
\end{ex}

A surprising fact is that we can actually have almost calibrated translators.

\begin{ex}\label{ex:translators} For $n> 2$, given $\psi\in(0,\pi)$ and $\phi_1,\ldots,\phi_{n-2}\in(0,\pi)$  such that  
\begin{equation*}
\phi_{n-1}:=\psi-(\phi_1+\ldots+\phi_{n-2})\in(0,\pi),
\end{equation*}
we let $\phi_n=0$ and let $\mathbf{\phi}=(\phi_1,\ldots,\phi_n)$.  For $n=2$, we let $\phi_1=\psi$ and $\phi_2=0$.

Joyce--Lee--Tsui \cite{JLT.Ex} constructed examples of embedded Lagrangian translators $L_{\mathbf{\phi}}(\psi)$ in $\C^n$ diffeomorphic to $\R^n$ whose Lagrangian angle takes values in $(0,\pi-\psi)$, so they are almost calibrated and exact.  Notice that by taking $\psi$ close to $0$ the oscillation of the Lagrangian angle of $L_{\mathbf{\phi}}(\psi)$ can be made arbitrarily small.  
The 1-form $\alpha$ defining the translating direction of $L_{\mathbf{\phi}}(\psi)$ is a multiple of $\rd x_{n}$, where the multiple goes to $0$ as $\psi\to 0$.

The blow-down of $L_{\mathbf{\phi}}(\psi)$ is $P_0\cup P_{\mathbf{\phi}}$ in the notation of \eqref{eq:plane}, so they are two multiplicity one Lagrangian planes intersecting along a line (given by the $x_n$-direction), with \emph{different} Lagrangian angles: one has angle $0$ whilst the other has angle $\psi$.
\end{ex}

\subsubsection{Harvey--Lawson $T^2$-invariant special Lagrangians}  In 
\cite{HarveyLawson.Calibrated},  Harvey and Lawson constructed special Lagrangians in $\C^n$ invariant under the action of the maximal torus in $\mathrm{SU}(n)$ for $n\geq 3$.  We only consider here the case $n=3$.

\begin{ex}\label{ex:HL}
Given $a_1,a_2,a_3>0$, we define
\begin{equation}\label{eq:HL}
L_1(a_1)=\{(z_1,z_2,z_3)\in\C^3\,:\,|z_1|^2-a_1=|z_2|^2=|z_3|^2,\,z_1z_2z_3\in [0,\infty)\}
\end{equation}
and similarly define $L_2(a_2)$ and $L_3(a_3)$ by cyclic permutations of $(1,2,3)$.  Then $L_j(a_j)$ is an embedded special Lagrangian in $\C^3$, with Lagrangian angle $0$, diffeomorphic to $\mathcal{S}^1\times\R^2$, and asymptotic at infinity to the special Lagrangian $T^2$-cone
\begin{equation}\label{eq:HL.cone}
C=\{(z_1,z_2,z_3)\in\C^3\,:\,|z_1|=|z_2|=|z_3|,\,z_1z_2z_3\in\R^+\}.
\end{equation} 
Each $L_j(a_j)$ is invariant under the $T^2$-action preserving $C$.\end{ex}

\noindent The special Lagrangians $L_j(a_j)$ are \emph{not} exact: in fact, the integral of $\lambda$ around the generator of the $\mathcal{S}^1$ in  $L_j(a_j)$ is proportional to $a_j$.  The cone $C$ is the simplest known non-planar special Lagrangian cone.

\section{Blow-downs of ancient solutions}\label{sect:blowdowns}

We show that blow-downs of zero-Maslov ancient solutions to LMCF in $\C^n$ are finite unions of special Lagrangian cones (with possibly different phases).  We further show that if the ancient solution is exact (and almost calibrated for $n>2$), then the blow-down 
of connected components of the flow  intersected with a ball is contained in a single  cone. This is the analogue of the structure theory of Neves \cite{Neves.ZM} for blow-ups of singularities of zero-Maslov LMCF, and the proofs 
are essentially the same.  Consequently, we only sketch the main points, and highlight any differences.

\subsection{Assumptions}\label{ss:assumptions}
Throughout we assume that $(L_t)_{-\infty<t<0}$ is a zero-Maslov ancient solution to LMCF in $\C^n$. We will assume that there is a bound on the area ratios and Lagrangian angle as $t\ra-\infty$;  that is, there exists a constant $A_0>0$ such that 
\begin{equation}
\limsup_{t\ra-\infty}\Ha^n(L_t\cap B_{R})\leq A_0 R^n\quad\text{and}\quad \limsup_{t\ra-\infty}\sup_{L_t}|\t_t|\leq A_0.
  \label{eq:bounds}
\end{equation}
 
By \eqref{eq:monotonicity} and \eqref{eq:theta.evol}, the bounds \eqref{eq:bounds} lead to uniform bounds on the area ratios and Lagrangian angle of $L_t$ for all $t$, see for example \cite[Lemma B.1]{Neves.ZM}.

We shall take a blow-down sequence $(L^i_s)_{-\infty<s<0}$ as in \eqref{eq:blow-down.seq} and let $\lambda^i_s=\lambda|_{L^i_s}$ and $\theta^i_s$ be the Lagrangian angle of $L_s^i$.

\subsection{Zero-Maslov structure theorem} 

We first recall a structure theorem for zero-Maslov blow-downs of ancient solutions to LMCF in $\C^n$.  For $n=2$ this is included in \cite[Theorem 3.1]{NevesTian}.  The result is analogous to \cite[Theorem A]{Neves.ZM} and the proof is almost identical; we include it for completeness.

\begin{theorem}\label{thm:blow-downA}
 Let $(L_t)_{-\infty<t<0}$ be a zero-Maslov ancient solution to LMCF in $\C^n$ satisfying \eqref{eq:bounds}.  For any blow-down sequence $(L_s^i)_{-\infty<s<0}$ as in \eqref{eq:blow-down.seq}, there exists $N\in\bb{N}$ and 
special Lagrangian cones $C_j$ with distinct Lagrangian angle $\ot_j$ for $j=1,\ldots,N$ 
 such that, after passing to a subsequence, for all $\phi\in C^{\infty}_c(\C^n)$, $f\in C^2(\bb{R})$ and $s<0$ we have
\begin{equation}\label{eq:blow-downA}
\underset{i\ra\infty} \lim \int_{L_s^i} f(\t^i_{s}) \phi\, \rd\Ha^n = \sum_{j=1}^N m_j f(\ov\t_j)\mu_j(\phi),
\end{equation}
where $\mu_j$ and $m_j$ are the Radon measures and multiplicity of the support of $C_j$.  Furthermore, the set $\{\ot_j:j=1,\ldots,N\}$ of angles does not depend on the sequence of rescalings chosen in \eqref{eq:blow-down.seq}.
\end{theorem}

\noindent Here, by a special Lagrangian cone we mean a   special Lagrangian integral current which is invariant under dilations.  This result states that any blow-down of $L_t$ is a union of special Lagrangian cones, and the Lagrangian angles of these cones are independent of the blow-down sequence (though the cones themselves could vary).

\begin{proof}[Proof of Theorem \ref{thm:blow-downA}] 
By the assumptions  \eqref{eq:bounds} we may apply 
 \eqref{eq:monotonicity} to  
 $L^i_s$, where we choose the functions $f=\theta^2$ and $f=1$.  Using the evolution equation \eqref{eq:theta.evol.2}, we can argue just as in \cite[Lemma 5.4]{Neves.ZM} that the following holds.
\begin{lem}\label{lem:integral.convergence}
For any $-\infty<a<b<0$ and any $R>0$, we have 
\begin{equation}\label{eq:integral.convergence}
\underset{i\ra \infty} \lim \int_a^b \int_{L^i_s\cap B_R } (|\lambda^i_s|^2 +|\rd\theta^i_s|^2)\,\rd\Ha^n\rd s=0.
\end{equation}
\end{lem}
We may now choose $a<0$ such that 
 \begin{equation}\underset{i\ra \infty} \lim \int_{L^i_a\cap B_R } (|\lambda^i_a|^2 +|\rd\theta^i_a|^2)\rd\Ha^n=0,
  \label{eq:asubs}
 \end{equation}
since this holds for almost all $a<0$ by  \eqref{eq:integral.convergence}.  Then \eqref{eq:asubs}, together with assumptions  \eqref{eq:bounds}, allows us to invoke a compactness result of Neves \cite[Proposition 5.1]{Neves.ZM}.

\begin{prop}\label{prop.compactness}
 Let $(L^i)$ be a sequence of zero-Maslov class Lagrangians in $\bb{C}^n$ such that, for some fixed $R>0$, 
there exists a constant $A_0>0$ so that
\begin{gather*}
\Ha^n(L^i\cap B_{2R}) \leq A_0R^n,\quad \underset{L^i\cap B_{2R} }\sup |\t^i|\leq A_0\qquad \forall i\in\bb{N},\\ 
 \underset{i \ra \infty} \lim \Ha^{n-1}(\partial L^i \cap B_{2R }) =0\quad\text{and}\quad
 \underset{i\ra \infty} \lim \int_{L^i\cap B_{2R} } |\rd\theta^i|^2 \rd\Ha^n =0.
 \end{gather*}
There exists $N\in\bb{N}$, special Lagrangians integral currents $C_j$ with Lagrangian angle $\ot_j$ for $j=1,\ldots,n$ such that, after passing to a subsequence, for all $\phi\in C^{\infty}_c(B_R )$ and $f\in C(\bb{R})$ we have
\begin{equation*}
\underset{i\ra \infty} \lim \int_{L^i} f(\t^i)\phi\, \rd \Ha^n = \sum_{j=1}^Nm_jf(\ot_j)  \mu_j(\phi),
\end{equation*}
where $\mu_j$ and $m_j$ are the Radon measure and multiplicity of the support of $C_j$.
\end{prop}

We deduce from Proposition \ref{prop.compactness}, after a diagonalisation argument, the existence of special Lagrangian integral currents $C_j$ with Lagrangian angle $\ot_j$ for $j=1,\ldots, N$ so that \eqref{eq:blow-downA} holds for $s=a$.  We also note that \eqref{eq:asubs} implies that $\lambda|_{C_j}\equiv 0$ and hence, by \eqref{eq:magic.formulae}, the $C_j$ are cones.  Furthermore, arguing just as in \cite{Neves.ZM}, we have that \eqref{eq:asubs} in fact must now hold for all $s<0$. 

Finally, the fact that the set of angles does not depend on the blow-down sequence follows from the monotonicity formula \eqref{eq:monotonicity} with function $(\theta_t-y)^{2m}$ for $y\in\R$ and $m\in\bb{Z}$, together with the evolution equation \eqref{eq:theta.evol.2}, just as in \cite[p.~22--23]{Neves.ZM}.
\end{proof}

\subsection{Exact and almost calibrated structure theorem} We now extend our structure theory for blow-downs of ancient solutions to LMCF to the exact and almost calibrated setting.  This is the analogue of \cite[Theorem 4.2]{Neves.Survey}, which is a refinement of \cite[Theorem B]{Neves.ZM} (but with the same proof), and again the proof of the result here is almost the same.

In addition to our standing assumptions in subsection \ref{ss:assumptions}, we will further assume the flow $(L_t)_{-\infty<t<0}$ is exact, which means there exist functions $\beta_t:L_t\to\R$ such that
\begin{equation}\label{eq:exact.assumption}
\lambda_t=\lambda|_{L_t}=\rd\beta_t,
\end{equation}
and that there exists $\epsilon>0$ such that the Lagrangian angles $\t_t$ satisfy
\begin{equation}\label{eq:almost.calibrated.bound}
\liminf_{t\ra-\infty}\cos\t_t\geq\epsilon,
\end{equation}
which implies that $L_t$ is almost calibrated for all $t$ by \eqref{eq:theta.evol}.  Moreover, \eqref{eq:almost.calibrated.bound} also holds for the angles $\theta^i_s$ of the blow-down sequence $L^i_s$, and the analogue of \eqref{eq:exact.assumption} holds for some functions $\beta^i_s$.

\begin{remark}
The almost calibrated condition \eqref{eq:almost.calibrated} is preserved along LMCF but we are not guaranteed in general a uniform $\epsilon$ as in \eqref{eq:almost.calibrated.bound} for an ancient solution with $L_t$ almost calibrated for all $t$.
\end{remark}

We may view $L_t$ as the images of immersions $\iota_t:L\to\C^n$ for a fixed  $L$.   

\begin{theorem}\label{thm:blow-downB}
Let $(L_t)_{-\infty<t<0}$ be an exact and zero-Maslov ancient solution to LMCF in $\C^n$ satisfying \eqref{eq:bounds}.  If $n>2$, we further assume that the almost calibrated condition  \eqref{eq:almost.calibrated.bound} holds.
Theorem \ref{thm:blow-downA} applies so we use the notation of that result. 

For almost all $s_0$, if $\Si^i \subset L_{s_0}^i$ has $\partial \Si^i \cap B_{3R} = \emptyset$ and only one connected component of $\Si^i\cap B_{2R}$ intersects $B_R$ then, after passing to a subsequence, we can find $j\in\{1,\ldots,N\}$ and $m\leq m_j$ so that, for all $\phi\in C^{\infty}_c(B_R )$ and $f\in C^2(\bb{R})$, we have
\begin{equation*}
\underset{i\ra \infty} \lim \int_{\Si^i} f(\t^i_{s_0}) \phi\, \rd \Ha^n = m f(\ot_j)\mu_j(\phi).
\end{equation*}
\end{theorem}

\begin{proof} 
Applying Lemma \ref{lem:integral.convergence}, we have that \eqref{eq:asubs} holds for $a=-1$ and $a=s_0$ without loss of generality.  This allows us to apply the following lemma (\cite[Lemma 3.10]{Neves.Survey}) to $L^i_{-1}$ and $\Si^i$ for all $i$ sufficiently large. 
\begin{lem}\label{lem:useful.area}
Suppose $L$ is a smooth Lagrangian in $\C^n$ with $\partial L\cap B_R=\emptyset$ and either $n=2$ and for some sufficient small $\delta>0$ we have, 
\begin{equation*}
\int_{L\cap B_R} |\rd\theta|^2 \rd\Ha^2 < \delta,
\end{equation*}
or $n>2$ and there exists $\epsilon>0$ so that
  \begin{equation*}
  \inf_{L\cap B_R}\cos\theta\geq\epsilon.
\end{equation*}   
There exists a constant $A_1=A_1(n,\e)>0$  such that
 \[\left(\Ha^n(U)\right)^\frac{n-1}{n} \leq A_1 \Ha^{n-1} (\partial U) ,\]
 for all open $U\subseteq L\cap B_R$ with rectifiable boundary.
\end{lem}  

\begin{remark}
The constant $A_1$ given by the proof of Lemma \ref{lem:useful.area} tends to infinity as $\e\ra 0$.  This motivates the need for the assumption \eqref{eq:almost.calibrated.bound}.
\end{remark}

Let $R_k>0$ be a monotonic sequence with $R_k\to\infty$. Applying Proposition \ref{prop.compactness} to $L^i_{-1}\cap B_{4R_k}$, as in the proof of \cite[Lemma 7.2]{Neves.ZM}, we deduce the convergence of connected components of $L^i_{-1}\cap B_{4R_k}$ intersecting $B_{R_k}$ to a union of special Lagrangian cones in $B_{2R_k}$.   By \eqref{eq:magic.formulae}, we get uniform bounds on $|\rd\beta^i_{-1}|$ and $|\rd\theta^i_{-1}|$.

Given these bounds and Lemma \ref{lem:useful.area},  we may apply a Poincar\'e-type Lemma \cite[Lemma 3.7]{Neves.Survey} to $L^i_{-1}$ with functions $\beta^i_{-1}$ and $\theta^i_{-1}$, which we give in our required form.
\begin{lem}\label{lem:Poincare}
Let $(L^i)$ and $(\a^i)$ be a sequence of smooth Lagrangians in $\bb{C}^n$ and smooth functions on $L^i$ such that for some $R>0$ the following holds: $L^i \cap B_{2R}$ is connected and $\partial(L^i \cap B_{3R})\subset B_{3R}(0)$ and there exists a constant $A_1>0$ so that
\begin{equation*}
\Ha^n(L^i\cap B_{3R})\leq A_1R^n
\quad\text{and}\quad
\left(\Ha^n(U)\right)^\frac{n-1}{n}\leq A_1 \Ha^{n-1}(\partial U)
\end{equation*}
 for all open $U\subseteq L^i\cap B_{3R}(0)$ with rectifiable boundary and all $i\in\bb{N}$;
 \begin{equation*}
 \sup_{L^i\cap B_{3R}}|\nabla\a^i|+R^{-1}\sup_{L^i\cap B_{3R}}|\a^i|\leq A_1 \; \text{ and }\;
\underset{i\ra \infty} \lim \int_{L^i\cap B_{3R}}|\n \a^i|^2 \rd\Ha^n =0.
\end{equation*}
Then, there exists $\ov\a\in\bb{R}$ such that after passing to a subsequence, we have
\begin{equation*}
\lim_{i\ra\infty}\sup_{x\in L^i\cap B_R}|\a^i(x)-\ov\a| =0.
\end{equation*}
\end{lem}

Lemma \ref{lem:Poincare} and Proposition \ref{prop.compactness} then imply there exist $\widehat{N}$, distinct pairs $(\ov\b_j,\ov\t_j)$ and special Lagrangian cones $\widehat{C}_j$ with Lagrangian angle $\ov\t_j$ for $j=1\ldots,\widehat{N}$ so that 
\begin{equation}\underset{i\ra\infty} \lim \int_{L_{-1}^i} f(\b^{i}_{-1} - 2(s_0+1) \t^i_{-1}) \phi\, \rd \Ha^n = \sum_{j=1}^{\widehat{N}} \widehat{m}_jf(\ov\b_{j} - 2(s_0+1) \ot_{j}) \widehat{\mu}_j(\phi) 
 \label{eq:-1.limit}
\end{equation}
for all $\phi\in C^{\infty}_c(\C^n)$ and $f\in C^2(\R)$, where $\widehat{m}_j$ and $\widehat{\mu}_j$ are the multiplicity and Radon measures of the support of $\widehat{C}_j$.
  We may further assume that the numbers $\ov\b_j-2(s_0+1)\ot_j$ are distinct.

Applying Proposition \ref{prop.compactness} to $\Si^i\cap B_{2R}$, we obtain a stationary (but a priori not special) Lagrangian cone $\Si$ with Radon measure $\mu_\Si$ so that $\Si^i$ converges weakly to $\Si$ on $B_R$.  We now argue that $\mathcal{H}^n(\Si)>0$.   There exists a sequence $x_i \in \Si^i\cap B_R$, which we may assume converges to some $x\in \ov B_R$.  Due to upper semicontinuity, the Gaussian density ratios of $\Sigma$ satisfy $\Theta(x,l)\geq 1$. Since $\Sigma$ may be considered as a stationary Brakke flow, we see that for some $\phi\in C^{\infty}_c(B_R(x))$, we have
\begin{equation*}
\int_\Si \Phi(x,1) \phi\, \rd\Ha^n>\frac{1}{2}
\end{equation*}
 by the monotonicity formula. Thus $\Ha^n(\Si)>0$ as claimed.
 
As $\Si^i\cap B_{2R}$ converges to a stationary cone, Lemmas  \ref{lem:integral.convergence}, \ref{lem:useful.area} and \ref{lem:Poincare} imply the existence of a constant $\ov\b$ so that 
\begin{equation*}
\lim_{i\ra\infty}\sup_{x\in \Si^i\cap B_R}|\beta^i_s(x)-\ov\b| =0.
\end{equation*}
We now let $\phi\in C^{\infty}_c(B_R)$ be nonnegative such that
\begin{equation}\label{eq:mu.positive}
\mu_\Si(\phi)=\int_\Si \phi\,\rd\mathcal{H}^n>0
\end{equation}
(this is possible since $\Ha^n(\Si)>0$) and let $f\in C^2_c(\R)$ be  such that $f:\R\to[0,1]$, 
\begin{equation}\label{eq:f.condns}
f(\ov\b)=1\quad\text{and}\quad f(\ov\beta_j-2(s_0+1)\ot_j)>0
\end{equation}
for at most one $j\in\{1,\ldots,\widehat{N}\}$.

From Lemma \ref{lem:integral.convergence} and the evolution  for $\beta^i_s+2(s-s_0)\t^i_s$ from \eqref{eq:beta.evol.2}, we see that
\begin{equation}\label{eq:s0.-1.limit}
\underset{i\ra \infty}\lim \int_{L_{s_0}^i} \b^i_{s_0} \phi \,\rd\Ha^n = \underset{i\ra \infty}\lim \int_{L_{-1}^i} \big(\b^i_{-1}-2(s_0+1)\t^i_{-1}\big)\phi d\Ha^n.
\end{equation}
We deduce from \eqref{eq:-1.limit}, \eqref{eq:mu.positive}, \eqref{eq:f.condns} and \eqref{eq:s0.-1.limit} that
\begin{align}\nonumber
0<\mu_\Si(\phi)&=\int_{\Si}f(\ov\b)\phi\,\rd\Ha^n=\underset{i\ra \infty}\lim \int_{\Si^i} f(\b^i_{s_0}) \phi \,\rd\Ha^n\leq
\underset{i\ra \infty}\lim \int_{L_{s_0}^i} f(\b^i_{s_0}) \phi \,\rd\Ha^n\\
& = \sum_{k=1}^{\widehat{N}} \widehat{m}_k f(\ov\b_k -2(s_0+1)\ot_k)\widehat{\mu}_k(\phi)=\widehat{m}_j f(\ov\b_j -2(s_0+1)\ot_j)\widehat{\mu}_j(\phi)
.\nonumber
\end{align}
By varying $f$ and $\phi$, we deduce that $\ov\b=\ov\b_j-2(s_0+1)\ot_j$ for some unique $j$ and hence $\Si=m \widehat{C}_j$ with $m\leq\widehat{m}_j$.  Moreover, since $\Si$ arises from a limit of $\Si^i\subseteq L^i_{s_0}$, we can relate $\widehat{C}_j$ and $\widehat{m}_j$ to some $C_j$ and $m_j$ as given by Theorem \ref{thm:blow-downA}. 
\end{proof}

\section{Stationary ancient solutions}\label{sect:stationary}

We show that for certain blow-downs to occur for a class of zero-Maslov ancient solutions to LMCF in $\C^n$, the flow must in fact be stationary.

\subsection{Planar blow-downs} We have the following analogue of \cite[Corollary 4.3]{Neves.Survey}.

\begin{prop}\label{prop:planar.blow-down}
Let $(L_t)_{-\infty<t<0}$ be an exact, zero-Maslov ancient solution to LMCF in $\C^n$ satisfying \eqref{eq:bounds}.  If $n>2$, we also assume the almost calibrated condition \eqref{eq:almost.calibrated.bound}.  
Suppose that a blow-down (given by Theorem \ref{thm:blow-downA}) is a union of multiplicity one planes $P_j$, which pairwise intersect transversely, and which have distinct Lagrangian angles $\ov{\t}_j$ for $j=1,\ldots,N$. Then $L_t = \cup_{j=1}^N P_j$ for all $t<0$.
\end{prop}

Since the proof of \cite[Corollary 4.3]{Neves.Survey} is only sketched, we provide the full proof of Proposition \ref{prop:planar.blow-down} in the case $n=N=2$ as the general case is proved in an entirely similar manner.  The idea is to use Theorem \ref{thm:blow-downB}, together with the non-existence of compact embedded exact Lagrangians in $\C^n$, to show that the intersection of the blow-down sequence with a large fixed ball eventually consists of exactly two components.  White's regularity theorem then yields a uniform curvature bound for the blow-down sequence in this ball.  Rescaling this curvature estimate forces the flow to be totally geodesic, and hence it must equal the union of planes. 

\begin{proof}[Proof of Proposition \ref{prop:planar.blow-down}]
For simplicity, we assume that $n=N=2$.  We begin with an elementary observation.

\begin{lem}\label{lem:density.1}
For all $\e_0>0$ there exists $R_0>0$ such that 
\begin{equation*}
\int_{P_1 +P_2} \Phi(x_0, l)(x,0)\,\rd\Ha^2\leq 1+ \frac{\e_0}{2} \quad \forall\; 0< l \leq 4\;\text{ and }\; |x_0|\geq \frac{R_0}{2}  .
\end{equation*}
 \end{lem}
\begin{proof}
Let $a_j= d(P_j,x_0)$.  Then 
 \begin{align}
\frac{1}{4\pi l}\int_{P_1} e^{-\frac{|x-x_0|^2}{4l}}\,\rd\Ha^2&+\frac{1}{4\pi l}\int_{P_2} e^{-\frac{|x-x_0|^2}{4l}}\,\rd\Ha^2\nonumber\\
  &=e^{\frac{-a_1^2}{4l}}\frac{1}{4\pi l}\int_{P_1} e^\frac{-|y|^2}{4l}\,\rd y+e^{-\frac{a_2^2}{4l}}\frac{1}{4\pi l}\int_{P_2} e^{-\frac{|y|^2}{4l}}\rd y\nonumber\\
  &=e^{-\frac{a_1^2}{4l}}+e^{-\frac{a_2^2}{4l}}.
  \label{eq:density.est}
 \end{align}
Since $l\leq 4$, the result follows.
\end{proof}

By \eqref{eq:Theta.dec},  for $-2\leq s<0$ and $0\leq l\leq 2$  
the Gaussian density ratios of $L^i_s$ satisfy
\begin{equation*}
\T_s^i(x_0,l) \leq \T_{-2}^i(x_0,l+2+s).
\end{equation*}
Thus, Radon measure convergence and Lemma \ref{lem:density.1} imply that 
for any $\e_0>0$ there is $R_0>0$ such that for all $|x_0|>\frac{R_0}{2}$ and $i$ sufficiently large, $\T_s^i(x_0,l) \leq 1+\e_0$. 

We may therefore apply White's regularity theorem \cite{White.Local.Reg} to obtain that (for any fixed $K>R_0$) in the annulus $\ov{A(R_0, K)}:=\ov B_K(0)\setminus B_{R_0}(0)$, the $C^{2,\alpha}$ norm of $L_s^i$ is uniformly bounded for all $s\in[-1,0)$. We now show that, in this annulus, $L_s^i$ has two components which are graphs over $P_1$ and $P_2$ for $i$ large.

\begin{lem}\label{lem.annulus.components} On $\ov{A(R_0,K)}$, for $i$ sufficiently large and $s\in[-1,0)$, a subsequence of $L^i_s\cap \ov{A(R_0,K)}$ may be decomposed into two connected components, $\Si_{1,s}^i$ and $\Si_{2,s}^i$ which are graphical over $P_1 \cap A(R_0,K)$ and $P_2 \cap A(R_0,K)$ respectively.
\end{lem}

\begin{proof}
  Suppose first that there exists a sequence of points $x^i\in L_s^i\cap\ov{A(R_0,K)}$ such that $d(P_1\cup P_2,x^i)>3\e$ for some $\e>0$. For $j=1,2$, let 
\begin{equation*}
S_j^\e: = \{ x\in \ov{A(R_0,K)} |\ d(x,P_j)<\e\}.
\end{equation*} Due to the uniform $C^{2, \a}$ estimates on $L^i_s\cap \ov{A(R_0,K)}$, there is a uniform (in $i$) neighbourhood of $x^i$ in $L_s^i$ that may be written graphically over $T_{x^i}L^i_s$ and is disjoint from a $2\e$-tubular neighbourhood of $P_1 \cup P_2$.  Hence, there is some $c>0$ so that 
\begin{equation*}
\Ha^2(L_s^i\setminus(S_1^\e\cup S_2^\e))>c,
\end{equation*} a contradiction to Radon measure convergence. Therefore for any $\e>0$, choosing $i$ large enough, $L^i_s\cap \ov{A(R_0,K)}$ is contained in the disjoint sets $S_1^\e$ and $S_2^\e$. 

Suppose for a contradiction that for small $\e>0$ to be determined (and depending only on the $C^{2,\alpha}$ estimate), $i$ has been chosen large enough so that $L_s^i \cap S_j^\e\neq\emptyset$ but $L_s^i \cap S_j^\e$ is not graphical over $P_1$. Then, there exists $x\in L_s^i \cap S_j^\e$ and a unit vector $v \in T_xL_s^i$ with $v\perp P_j$. The uniform $C^{2, \a}$ bound implies that $L_s^i$ may be written graphically over a ball of radius $r$ in $T_x L_s^i$ (where $r$ depends only on the $C^{2,\a}$ estimate). However, choosing $\e = \frac{r}{2}$, we see that in the $v$ direction $L_s^i$ leaves $S_j^\e$, a contradiction. 

Hence, $L_s^i \cap S_j^\e$ is graphical over $P_j$ and the $C^0$ norm of the graph converges to zero as $i\ra \infty$. Therefore, using Ehrling's Lemma and the uniform $C^{2, \alpha}$ bound we see that $L_s^i \cap S_j^\e$ is made up of $k$ graphical ``sheets'' that converge in $C^2$ to the plane $P_j$.  Due to Radon measure convergence to multiplicity one planes we see that there must exist a subsequence of the $L^i_s\cap S_j^\e$ which is a graph over $P_j$; that is, $k=1$. 

Setting $\Si_{j, s}^i = L_s^i \cap S_j^\e$ gives the result.
\end{proof}

We now show that $L^i_s$ has two connected components in the ball, as well as in the annulus.

\begin{lem}\label{lem:component.convergence}
There is a subsequence of $L^i_s$, such that for $i$ large enough there are exactly two connected components of $L^i_s \cap B_K$, $\Si_{1,s}^i$ and $\Si_{2,s}^i$, where $\Si_{j,s}^i$ converges to $P_j\cap B_K$ for almost all $s \in [-1,0)$ in the sense of Radon measures.
\end{lem}

\begin{proof}
Theorem \ref{thm:blow-downB} implies that $L^i_s \cap B_K$ must have at least two connected components for almost all $s\in[-1,0)$, as otherwise we would conclude that $\ov{\t}_1 = \ov{\t}_2$. 

Now let $\Si_{j,s}^i$ be a connected component that converges to $P_j$ on the annulus $A(R_0,K)$, as in Lemma \ref{lem.annulus.components}. Due to Theorem \ref{thm:blow-downB}, $\Si_{j,s}^i$ must converge (as a Radon measure) to a subset of either $P_1$ or $P_2$.  Since it converges in $C^2$ on $A(R_0, K)$ to $P_j$, we deduce that $\underset{i\ra\infty}\lim\Si_{j,s}^i\subseteq P_j\cap B_K$. Due to the isoperimetric inequality, $\bb{M}(\Sigma_{j,s}^i)\geq \pi K^2$ (where here $\bb{M}$ is the mass of $\Sigma_{j,s}^i$ considered as a varifold), and so since $\Sigma_{1,s}^i$ and $\Sigma_{2,s}^i$ are disjoint connected components,
\[\bb{M}(L^i_s\cap B_K)\geq \bb{M}(\Sigma_{1,s}^i\cup\Sigma_{2,s}^i) = \bb{M}(\Sigma_{1,s}^i)+\bb{M}(\Sigma_{2,s}^i)\geq 2\pi K^2.\]
By convergence in Radon measure, we have that for all $s\in[-1,0)$, \[\lim_{i\ra\infty}\bb{M}(L^i_s\cap B_K) = 2 \pi K^2.\]  Let $\iota^i_s:L\to\bb{C}^n$ be the immersion of $L^i_s$ and let \[\tilde{L}_s^i = \iota^i_s\left[L\setminus\left( (\iota^i_s)^{-1}(\Sigma^i_{1,s})\cup (\iota^i_s)^{-1}(\Sigma^i_{2,s})
\right)\right]\cap B_K.\] 
By subadditivity, \begin{equation}\label{eq:mass.limit.zero}
\lim_{i\ra\infty}\bb{M}(\tilde{L}_s^i)=0.
\end{equation}
Due to Lemma \ref{lem.annulus.components} we also have that $\tilde{L}_{s}^i \subset B_{R_0}$ for all $s\in[-1,0)$. 

We   now use the clearing out lemma \cite[Proposition 4.23]{Ecker.MCF.Book}
 to see that $\tilde{L}_{s}^i=\emptyset$.  For a contradiction, suppose not.  Let $\epsilon>0$ and choose $\rho$ such that $\rho^2=2n\epsilon$.  Let $\kappa(n,\frac{1}{4n})>0$ be the constant given in \cite[Proposition 4.23]{Ecker.MCF.Book}.   By \eqref{eq:mass.limit.zero} we may choose $s_0\in(-1,-1+\frac{1}{2}\epsilon)$ and $j$ sufficiently large so that
\[\rho^{-n}\bb{M}(\tilde{L}_{s_0}^j \cap B_\rho(x_0))<\kappa\]
for any $x_0\in B_{R_0}$.  
By \cite[Proposition 4.23]{Ecker.MCF.Book}, $\tilde{L}_{s}^j=\emptyset$ for all $s\in[-1+\e,0)$.  Since $L^i_s$ is a blow-down sequence we may assume  there exists $\sigma>2$ such that $\tilde{L}^i_s=\sigma^{-1}\tilde{L}^j_{\sigma^{-2}s}$ for all $s\in[-1,0)$.  This implies $\tilde{L}^i_s=\emptyset$, our required contradiction.

Finally, due to the isoperimetric inequality and Theorem \ref{thm:blow-downB}, we see that $\Sigma_{j,s}^i$ must converge to $P_j\cap B_K$ as a Radon measure for almost all $s\in[-1,0)$.
\end{proof}

Since $\Si_{j,s}^i$ becomes arbitrarily close in the Radon measure sense to the plane $P_j$ by Lemma \ref{lem:component.convergence}, we may apply a variant of White's regularity theorem which may be localised to each $\Si_{j,s}^i$ (c.f.~\cite[Remark 4.16 (4), Proposition 4.17, Lemma 5.8]{Ecker.MCF.Book}). Hence, the $\Si^i_{j,s}$ satisfy a uniform $C^{2,\alpha}$ bound in $B_{\frac{R_0}{2}}$ for $-\frac{1}{2}\leq s<0$. In particular   $L_s^i\cap B_\frac{R_0}{2}$ satisfies  $|A^i_s|<C$ for $s\in[-\frac{1}{2},0)$, where $A^i_s$ is the second fundamental form of $L^i_s$. By scaling this implies the second fundamental form $A_t$ of $L_t$ satisfies 
\begin{equation*}
|A_t|<\frac{C}{\s_i}
\end{equation*}
for $-\frac{\s_i^2}{2}\leq t<0$. Sending $i\ra \infty$,  we see that $A_t=0$ and thus $L_t=P_1\cup P_2$.
\end{proof}

\subsection{Special Lagrangian blow-downs}  We show that, if $(L_t)_{-\infty<t<0}$ is an ancient solution to LMCF in $\C^n$ satisfying \eqref{eq:bounds} which is almost calibrated, and some blow-down is a special Lagrangian cone (meaning that it has just one Lagrangian angle), then $L_t$ is itself special Lagrangian (and thus stationary).

\begin{prop}\label{prop:SL.blowdown}
Let $(L_t)_{-\infty<t<0}$ be an ancient solution to LMCF in $\C^n$ which is zero Maslov and satisfying \eqref{eq:bounds}. Assume further that the Lagrangian angle satisfies
\begin{equation}\label{eq:angle.restrict}
\theta_t \in [\varepsilon, 2\pi - \varepsilon]
\end{equation}
 for some $\varepsilon >0$.  Suppose that for some blow-down sequence $(L^i_s)_{-\infty<s<0}$ as in \eqref{eq:blow-down.seq} we have that $(L^i_s)_{-\infty<s<0}$ converges to a special Lagrangian cone $C$ with Lagrangian angle $\ov\t$ as varifolds as $i\ra\infty$.  Then $(L_t)_{-\infty<t<0}$ is special Lagrangian with Lagrangian angle $\ov\t$.
\end{prop}

\begin{proof}
The monotonicity formula \eqref{eq:monotonicity} together with the evolution equation \eqref{eq:theta.evol.2} for the Lagrangian angle $\t_t$ implies that
\begin{equation}\label{eq:angle.decay}
\begin{split}
\ddt{} \int_{L_t} |\theta_t -\ov\theta|^2 \Phi_{(0,1)}\, \rd\Ha^n =&\  
-  \int_{L_t} |\rd\t_t|^2 \Phi_{(0,1)}\, \rd\Ha^n\\
&\   -  \int_{L_t} 
|\theta_t -\ov\t|^2 \left|\rd\t_t + \frac{\lambda_t}{2t}\right |^2 \Phi_{(0,1)}\, \rd\Ha^n\, ,
\end{split}
\end{equation}
where $\bar{\theta} \in [\varepsilon, 2\pi - \varepsilon]$. Note that by assumption we have a positive sequence $\sigma_i\to\infty$ such  that for any $t<0$
\begin{equation}\label{eq:varifold.cone.convergence}
\sigma_i^{-1}L_{\sigma_i^{2}t} \ra C\
\end{equation}
 in the varifold sense.  We now consider the (oriented) angle function $\theta$ on the Grassmanian of oriented Lagrangian planes in $\C^n$, where we choose the values of $\theta$ to lie in $[0,2\pi)$. Note that $\theta$ is not continuous, but we can define a function $\psi: [0,2\pi) \rightarrow [\varepsilon, 2\pi - \varepsilon]$ by
$$ \psi(\theta)= 
\begin{cases} \ \varepsilon - \frac{\pi - \varepsilon}{\varepsilon}(\theta - \varepsilon) \qquad \qquad \qquad &\text{for}\ \theta \in [0,\varepsilon),\\
\ \theta &\text{for}\ \theta \in [\varepsilon, 2\pi - \varepsilon],\\
\ 2\pi-\varepsilon  - \frac{\pi - \varepsilon}{\varepsilon}(\theta - (2\pi - \varepsilon)) &\text{for}\ \theta \in (2\pi-\varepsilon, 2\pi).
\end{cases} 
$$
Note that  $ \psi(\theta) = \theta$ for  $ \theta \in [\varepsilon, 2\pi - \varepsilon]$ and $\psi(\theta)$ is a continuous function on the Grassmanian of oriented Lagrangian planes in $\C^n$. 

Since $L_t$ satisfies \eqref{eq:angle.restrict} we can replace $\theta_t$ by $\psi(\theta_t)$ in \eqref{eq:angle.decay}. The convergence \eqref{eq:varifold.cone.convergence} as varifolds and the decay of the heat kernel implies that
 \begin{equation}\label{eq:angle.infinity.limit}
 \lim_{t \ra -\infty} \int_{L_t}  |\theta_t -\ov\theta|^2 \Phi_{(0,1)}\, \rd\Ha^n  = \lim_{t \ra -\infty} \int_{L_t}  |\psi(\theta_t) -\ov\theta|^2 \Phi_{(0,1)}\, \rd\Ha^n = 0.
\end{equation}  
Combining \eqref{eq:angle.decay} and \eqref{eq:angle.infinity.limit}, we have $\theta_t \equiv \ov\theta$ as required.
\end{proof}

\begin{remark}
 The product $L$ of a grim reaper curve with $\R$ as in Example \ref{ex:grim} is an ancient (in fact, eternal) solution to LMCF in $\C^2$ which is not special Lagrangian, but
whose blow-down is a multiplicity two special Lagrangian plane for $s<0$, a multiplicity two special Lagrangian half-plane for $s=0$ and vanishes for $s>0$.  This shows that Proposition \ref{prop:SL.blowdown} is sharp since the Lagrangian angle of $L$ satisfies $0<
\theta<2\pi$ but it does not satisfy \eqref{eq:angle.restrict}.
\end{remark}

\section{Asymptotics for minimal submanifolds}\label{sect:asymptotics}

We obtain asymptotics for minimal submanifolds with planar blow-downs via an optimal \Lo--Simon inequality, and methods from \cite{SimonCone}. 
 As a consequence, we strengthen the uniqueness result for Lawlor necks in \cite{IJO.Uniqueness.Lawlor}.  We believe that our results here could also be proved using methods from \cite{AlmgremAllardCone}.

\subsection{Main results}

Throughout this section we let $L$ be a minimal $n$-dimensional submanifold of $\R^{n+k}$ with bounded area ratios, i.e.~satisfying \eqref{eq:bounded.area.ratios}.  Since the area ratios are bounded, we can consider a tangent cone $L^{\infty}$ of $L$ at infinity, and we let $\Sigma = L^\infty\cap \partial B_1$.  We assume that $L^{\infty}$ is a multiplicity one minimal cone that is smooth outside the origin, which ensures that $L^{\infty}$ is unique (by \cite{SimonCone}), and that $L^{\infty}$ is \emph{Jacobi integrable}, i.e.~the Jacobi fields on $\Sigma$, as a minimal submanifold of $\partial B_1$, are integrable. Our main result here is the following, where $r$ denotes the distance from $0$ in $\R^{n+k}$.

\begin{theorem}\label{decayestimate}
Let $L^n\subseteq\R^{n+k}$ be a minimal submanifold with bounded area ratios. Suppose that a blow-down $L^\infty$ of $L$ is a multiplicity one cone which is smooth away from $0$ and Jacobi integrable. 
There exists $R>0$ such that $L\setminus B_R$ may be written as a graph of a normal vector field $v$ over $L^\infty$ with $|v|\leq C r^\a$ for some $\a\in(0,1)$ and $C>0$.
\end{theorem}
\begin{remark}
If we drop the Jacobi integrable assumption on $L^{\infty}$ we obtain the same statement but with the weaker decay estimate $|v|<C\frac{r}{\log^pr}$ for some $p>0$.
\end{remark}

Since planes are Jacobi integrable (implicitly shown in \cite{AlmgremAllardCone}, for example), Theorem \ref{decayestimate} and the work in \cite{IJO.Uniqueness.Lawlor}, as explained in Section \ref{sect:Primliminaries}, implies we can strengthen the known  uniqueness statement for asymptotically planar special Lagrangians.

\begin{theorem}\label{thm:SL.planar.uniqueness}
Let $L$ be a smooth, exact, special Lagrangian in $\C^n$ with bounded area ratios.  Suppose that a blow-down of $L$ is a pair of transversely intersecting multiplicity one planes.  Then, up to rigid motions, $L$ is either a Lawlor neck as in Example \ref{ex:Lawlor} or the pair of planes.
\end{theorem}

\begin{remark} This shows that the uniqueness statement one obtains for $n=2$ using complex geometry, as explained in Section \ref{sect:Primliminaries}, extends to all dimensions.
\end{remark}

\subsection{Monotonicity}
Since  
$L^n\subseteq\R^{n+k}$ is minimal, if $\mu$ is the measure associated to $L$ and $D^{\perp}$ is the projection of the Euclidean derivative to the normal bundle of $L$, we have the following monotonicity formula:
\begin{equation}\frac{d}{d\rho}\left(\frac{\mu(B_\rho)}{\rho^n}\right) = \int_{L^n\cap B_\rho} \frac{|D^\perp r|^2}{r^{n}}d \Ha^n.
 \label{monotonicity}
\end{equation}
From \eqref{monotonicity} and \eqref{eq:bounded.area.ratios}, we have
\begin{equation} \mu_\infty := \lim_{\tau\ra\infty}\frac{\mu(B_\tau(0))}{\tau^n}
\label{tinfinity}
\end{equation}
is finite.  Furthermore, we have
\begin{equation}\int_{L^n\cap(\bb{R}^{n+k}\setminus B_\rho)}\frac{|D^\perp r|^2}{r^n}d\Ha^n = \mu_\infty - \frac{\mu(B_\r(0))}{\r^n}.
 \label{monotonicity2}
\end{equation}
 Equation \eqref{monotonicity} and the co-area formula imply that for almost all $\r \in (0,\infty)$,
\begin{equation}n\r^{-(n+1)}\mu(B_\rho) = \r^{-n}\int_{\partial B _\r \cap L}|\n r| d\Ha^{n-1} .\label{monotonicitydiff}
\end{equation}

\subsection{Energy} Following the notation in \cite[Section 7]{SimonCone}, we say that $L$ is a graph of $h\in C^{\infty}((TL^\infty)^{\perp})$ in the annulus $B_\r\setminus \overline{B}_\sigma$ if and only if
\begin{equation}G_{\sigma, \r}=\left\{\frac{x+h(x)}{\sqrt{1+\frac{|h(x)|^2}{|x|^2}}}: x\in L^\infty\cap(B_\r\setminus \overline{B}_\sigma)\right\}=L\cap(B_\r\setminus \overline{B}_\sigma). 
 \label{asgraph}
\end{equation}
In practice, it easier to work with logarithmic polar coordinates and view $h$ in terms of a normal vector field $u$ on $\Sigma\times (\log\s,\log\r)$:
\begin{equation} u(\omega, t) = r^{-1} h(\omega r), 
 \label{utoh}
\end{equation}
where $t=\log r$. We will write the restriction of $u$ to $\S$ at fixed $t$ as $u(t)(\cdot) =u(\cdot, t)$. 

For a section $v$ of the pullback of $(TL^{\infty})^{\perp}$ to $\Sigma$ of small $C^1$ norm, we define an energy via
\begin{equation}\mathcal{E}(v) = \int_\S J(\o,v,\n v)d\Ha^n(\o),
 \label{energygraph}
\end{equation}
where $J$ is chosen, as in \cite[Section 7]{SimonCone}, so that if $\|u(t)\|_{C^1(\Sigma)}$ is sufficiently small, 
\[\mathcal{E}(u(t)) = e^{t(1-n)}\Ha^{n-1}(L\cap \partial B_{e^t}) =
\Ha^{n-1}(e^{-t}L\cap \partial B_{1})  .\]
The energy $\mathcal{E}$ is uniformly convex for $\|v\|_{C^1(\Sigma)}$   sufficiently small. Observe that
\begin{equation}\mathcal{E}(0) = n\mu_\infty  .
 \label{energydensity}
\end{equation}
We define the elliptic operator $\mathcal{M} = -\text{grad}\,\mathcal{E}$ on $\S$.  The Jacobi integrability of $L^{\infty}$ implies we may apply \cite[Lemma 1 on p.~80]{SimonBook} to obtain an optimal \emph{\Lo--Simon inequality}, namely there exists $\s>0$ such that if $|u|_{C^3(\S)}<\s$ then   
 \begin{equation}|
 \mathcal{E}(u)-\mathcal{E}(0)|\leq C \|\mathcal{M}(u)\|^2_{L^2(\Sigma)}=C \|\text{grad}\,\mathcal{E}(u)\|^2_{L^2(\Sigma)}.
  \label{optimalLS}
 \end{equation}
 
 An advantage of our coordinates is that minimality becomes a uniformly  elliptic equation (cf.~\cite[p.~565]{SimonCone}): if $\|u\|_{C^1(\Sigma\times [1,\infty))}$ is small then $H\equiv 0$ if and only if
\begin{equation}\ddot u(t) + n \dot u(t) +\mathcal{M}(u)(t) + \mathcal{R}(u)(t) =0,
 \label{elliptic}
\end{equation}
where 
\[\mathcal{R}(u)(t) = (a\cdot \n^2 u (t) + b)\dot u (t) + c\cdot \n \dot u(t) + d \ddot u (t)  \]
and $a,b,c,d$ are smooth functions which vanish when $u$, $\nabla u$ and $\dot{u}$ are $0$.

We will need suitable estimates on $|D^\perp r|^2$. 
Calculating as in \cite[Page 561]{SimonCone}, if $\|u\|_{C^1(\Sigma)}$ is sufficiently small then we may estimate 
\begin{equation} \label{Dperprdt}\frac{1}{2}\left|e^t \dot u \right|^2 \leq |D^\perp r|_x|^2\leq {2}\left|e^t \dot u \right|^2,
\end{equation}
and
\begin{equation} \label{Drperpintest}
\begin{split}
 \frac 1 2 \int_{\log \r}^{\log{\sigma}} \int_{\Sigma} |\dot u |^2 d\Ha^{n-1}dt \leq&\ \int_{L\cap(B_\sigma \setminus \ov B_\r)}\frac{|D^\perp r |^2}{r^n}d\Ha^n\\
 \leq&\ 2 \int_{\log\r}^{\log{\sigma}} \int_{\Sigma} |\dot u |^2 d\Ha^{n-1}dt\, .
 \end{split}
\end{equation}

\subsection{Extension lemma} We now prove the following extension lemma.
 \begin{lem}\label{extension}
 Suppose $\beta, \tau>1$. For every $\sigma>0$, there exists $\d(\sigma,\beta,\Sigma)>0$ such that if $\mu_{\infty}$ in \eqref{tinfinity} satisfies
 $n\mu_\infty = \Ha^{n-1}(\Sigma)$ and $L\cap (B_{e^{ \tau+1}}\setminus\ov{B_{e^\tau}})$ is a graph of $u$ on $\Sigma \times (\tau, \tau+1)$ such that  
 \[\|u\|_{C^{3}(\S\times(\tau,\tau+1))}<\sigma \text{ and } \sup_{t\in(\tau, \tau+1)}\|u(t)\|_{L^2(\Sigma)}\leq \delta,\]
 then $L\cap (B_{e^{\tau+\beta}}\setminus\ov{B_{e^\tau}})$ is a graph of $\tilde{u}$, which is an extension of $u$ such that
 \[\|\tilde{u}\|_{C^{3}(\S\times(\tau,\tau+\beta))}<\sigma  .\]
\end{lem}

\begin{proof}
By changing scale, we may assume that $\tau=0$. Furthermore, \eqref{monotonicity2} implies
\begin{equation}\Ha^n(L \cap B_{e^{\beta}})\leq c_1(\mu_{\infty},\beta).
 \label{areabound}
\end{equation}

Suppose the lemma is false. Then there exists a sequence of minimal $L^k$ with $L^k\cap (B_e\setminus \ov{B_1}) = G_{1,e}(h^k)$ as in \eqref{asgraph} 
such that $\|r^{-1}h^k\|_{C^2(B_e\setminus \ov{B_1})}\ra 0$ and \eqref{areabound} holds,  but $L^k\cap(B_4\setminus \ov{B_1})$ cannot be written as $G_{1,e^\beta}(\tilde{h})$ for any $\tilde{h}$ such that $\|\tilde u\|_{C^3}<\s$, where $\tilde{u}$ and $\tilde{h}$ are related as in \eqref{utoh}.  By \eqref{areabound} we may take a subsequence (also written $L^k$) which converges to some varifold $V$. On $B_e\setminus \ov{B_1}$, $V$ must have density equal to 1 and support equal to $L^\infty$. Due to the convergence of the $h^k$,  \[\left(\frac{3}{2}\right)^{1-n}\Ha^{n-1}(L^k \cap \partial B_\frac{3}{2})\ra \Ha^{n-1}(\Sigma) = n\mu_\infty,\] and $V$ equals $L^\infty$ on $B_e\setminus \ov B_1$. However, due to  \eqref{monotonicity2} and \eqref{monotonicitydiff},
\[\int_{(\bb{R}^{n+k}\setminus B_\frac{3}{2})\cap L^k}\frac{|D^\perp r|^2}{r^n}d\Ha^n = \mu_\infty - \frac{1}{n}\left(\frac{3}{2}\right)^{1-n}\int_{\partial B_\frac{3}{2}\cap L^k} |\n r| d\Ha^{n-1}\]
and so we see that 
\[\int_{V\cap(\bb{R}^{n+k}\setminus B_\frac{3}{2})}\frac{|D^\perp r|^2}{r^n}d\Ha^n =0,\]
i.e.~$V$ is a unit density cone. Allard's regularity theorem (see, e.g.~\cite[Chapter 5]{Simon.Lectures}) then implies that for sufficiently large $k$, $L^k$ is a graph on $B_{e^\beta}\setminus\ov B_1$, with arbitrarily small $C^3$ norm, a contradiction.
\end{proof}

\subsection{Growth estimates} We want to estimate the growth of
\[I(\r) := \int_{(\bb{R}^{n+k}\setminus B_{e^\r})\cap L} \frac{|D^\perp r|^2}{r^n }d\Ha^n.\]
Using the \Lo--Simon inequality \eqref{optimalLS}, we obtain the following.
\begin{lem}\label{growthest1}
 There exists a constant $\sigma(\S)>0$ such that if $L\cap(B_{e^{\r+1}}\setminus \ov{B_{e^\r}})$ is the graph of $u$ with $|u|_{C^3(\Sigma\times [\r,\r+1])}<\s$ then there is a constant $C(\Sigma)$ so that
 \[I(\r+1)\leq C(I(\r)-I(\r+1)).\]
\end{lem}
\begin{proof}
Applying \eqref{monotonicitydiff} to \eqref{monotonicity2}, and using \eqref{energydensity} we see that for almost every $\r$,
\begin{flalign*}
n\int_{(\bb{R}^{n+k}\setminus B_{e^\r})\cap L}\frac{|D^\perp r |^2}{r^n}d\Ha^n &= n\mu_\infty - \r^{1-n}\int_{\partial B_{e^\r}\cap L} |\n r|d\Ha^{n-1}\\
&= \mathcal{E}(0) - \mathcal{E}(\r) + \r^{1-n}\int_{\partial B_{e^\r}\cap L} 1-|\n r|d\Ha^{n-1}\\
&\leq |\mathcal{E}(0) - \mathcal{E}(\r)|+\r^{1-n}\int_{\partial B_{e^\r}\cap L} |D^\perp r|^2 d\Ha^{n-1}  .
\end{flalign*}
We now choose $\s$ sufficiently small so that \eqref{optimalLS}--\eqref{Drperpintest} hold. We see that
\begin{flalign*}
n\int_{(\bb{R}^{n+k}\setminus B_{e^\r})\cap L}\frac{|D^\perp r |^2}{r^n}d\Ha^n&\leq \|\mathcal{M}(u(\r))\|^2_{L^2(\Sigma)}+\r^{1-n}\int_{\partial B_{e^\r}\cap L} |D^\perp r|^2 d\Ha^{n-1}\\
&\leq \int_{\Sigma}|\mathcal{M}(u(\r))|^2 + C|\dot u(\r)|^2 d\Ha^{n-1}\\
&=\int_{\Sigma}|\ddot u(\r)|^2+ |\dot u(\r)|^2 + \mathcal{R}^2(u(\r)) + 2|\dot u(\r)|^2 d\Ha^{n-1}\ .
\end{flalign*}
Now (possibly making $\s$ smaller) we may finally estimate the $\mathcal{R}$ term to give
\begin{flalign}\label{eq:growth.ineq}
\int_{L\cap(\bb{R}^{n+k}\setminus B_{e^\r})}\frac{|D^\perp r |^2}{r^n}d\Ha^n\leq C\int_{\Sigma}|\ddot u(\r)|^2+ |\nabla \dot u(\r)|^2 + |\dot u(\r)|^2 d\Ha^{n-1}\ .
\end{flalign}

If we differentiate \eqref{elliptic} with respect to $t$ we see that $\mathcal{L}\dot{u}=0$ where $\mathcal{L}$ is linear and uniformly elliptic.  Hence, we may use elliptic estimates in \eqref{eq:growth.ineq} and inequality \eqref{Drperpintest} to yield the claimed result.
\end{proof}

We have the following corollary of Lemma \ref{growthest1}.
\begin{cor}\label{growthest2}  Let $\sigma$ be as in Lemma \ref{growthest1}.
 Suppose that for $0<\r<\tau$, $$\sup_{t\in [\r,\tau]}\|u(t)\|_{C^3(\Sigma)}<\s.$$
 There exist positive constants $c_1$, $c_2$, $C$ depending on $\Sigma$ such that for all $\kappa\in[\r,\tau]$
\begin{align}
  I(\kappa )&\leq Ce^{c_1(\r-\kappa)}I(\r), \label{firstineq}\\
  \|u(\kappa) - u(\tau)\|_{L^2(\S)}&\leq Ce^{c_2(\r-\kappa)}\sqrt{I(\r)}\label{secondineq}\ .
\end{align}
\end{cor}

\begin{proof}
Suppose $0<a<b\leq 1$ are real numbers so there is a constant $c$ such that $a\leq c(b-a)$. It is elementary to deduce that $\log \left(b/a\right) \geq \log(1+c)$.  
This inequality and Lemma \ref{growthest1} quickly yield \eqref{firstineq}. 
 The second inequality \eqref{secondineq} follows by integration by parts applied to $\int_\kappa^\tau e^{\frac{c_1}{2} (t-\r)} \|\dot u(t)\|_{L^2(\Sigma)}^2 dt$. 
\end{proof}

\subsection{Proof of Theorem \ref{decayestimate}}
Pick $\s, C, c_2$ as in Corollary \ref{growthest2}, and then pick $\delta$ as in Lemma \ref{extension} where we choose $\beta = 2$. Using that $I(\r)$ is nonincreasing, we choose $\r_1$ sufficiently large so that for all $\r>\r_1$, $\sqrt{I(\r)}<\frac{\d}{4C}$. 
 
 Due to the convergence of the rescalings of $L$ to the blow-down $L^\infty$, there exists a rescaling $L^i$ so that in the annulus $B_{e^2}\setminus \ov{B_e}$, $L^i$ and $L^\infty$ are arbitrarily close in Radon measure. Applying Allard's regularity theorem (see e.g.~\cite[Chapter 5]{Simon.Lectures}), we can ensure that $L^i$ has arbitrarily small $C^3$-norm as a graph of $h$ over $L^\infty$. Rescaling back to $L$, we see that for all $\e>0$ there exists $\r_2>\r_1$ such that $L\cap(B_{e^{\r_2+1}}\setminus \ov{B_{e^{\r_2}}})$ is a graph of $u$ in logarithmic polar coordinates on $\Sigma\times[\r_2, \r_2+1]$ such that $\|u\|_{C^3(\S)}<\e$. We now choose $\e<\sigma$ sufficiently small so that $\|u(\r_2)\|_{L^2(\Sigma)}<\frac{\delta}{2}$.
 
 We are now able to apply Lemma \ref{extension} to extend $u$ to $\tilde{u}$ on $\Sigma\times[\r_2, \r_2+3]$ such that $\|\tilde{u}\|_{C^3(\Sigma\times(\r_2, \r_2+3))}<\s$. Moreover, \eqref{secondineq} implies that, for any $t\in(\r_2, \r_2+3)$,
 \[\|\tilde{u}(t)\|_{L^2(\Sigma)}\leq \|u(\r_2)\|_{L^2(\Sigma)}+ Ce^{c_2(\r_2-t)}\sqrt{I(\r_2)}<\frac{\d}{2}+\frac{\d}{2}=\d.\]
 We may therefore  apply Lemma \ref{extension} again and iterate to see that $u$ may be extended to $\Sigma\times(\r_2, \infty)$, such that $\|u\|_{C^3(\Sigma\times(\r_2, \infty))}<\s$. 
 
 We now obtain asymptotics for $u$. By Arzel\`a--Ascoli and Schauder estimates, there exists a sequence $t_i\ra\infty$ such that $u(t_i)\ra w$ smoothly. We deduce from \eqref{secondineq} that  $u(t)\ra w$ smoothly (since $I(\r)\ra 0$ as $\r\ra\infty$). Hence $w\equiv 0$, as otherwise the blow-down of $L$ would not be $L^\infty$. Taking the limit in \eqref{secondineq} gives 
 \[\|u(t)\|_{C^0(\Sigma)}\leq C\|u(t)\|_{L^2(\Sigma)}\leq Ce^{-c_2t}.\]
Recalling that $r = e^t$, we see that $|h|\leq Cr^{1-c_2}$ which implies Theorem \ref{decayestimate}.

\section{Special Lagrangian surfaces}\label{sectionSLsurfaces}

We saw in Lemma \ref{lem:SL.cplx} that special Lagrangians $L$ in $\C^2$ are hyperk\"ahler rotations of holomorphic curves.  
Here we show that if $L$ has bounded area ratios (i.e.~satisfies \eqref{eq:bounded.area.ratios}) then it is a hyperk\"ahler rotation of an \emph{algebraic curve}. 
 This  implies that the structure of the blow-down of $L$ (which is a union of planes) determines an explicit, finite-dimensional set of possibilities for $L$ given by zero sets of polynomials, and we can bound the total curvature of $L$ by a quadratic in $d$.
 
 Our methods further allow us to classify special Lagrangians with bounded area ratios in $\C^2$ with a blow-down given by two planes counted with multiplicty.

\subsection{Algebraic curves}

An algebraic curve 
$\Cu_P\subseteq\bb{C}^2$ is the zero set of the polynomial $P(x,y)$, that is
\[\Cu_P=\{(x,y)\in \bb{C}^2: P(x,y)=0\} ,\]
where, for some $c_{ij}\in\C$, 
\begin{equation}\label{eq:P.poly}
P(x,y) = \sum c_{ij}x^iy^j.
\end{equation}
We say that a non-constant polynomial $P$ is \emph{irreducible} if it cannot be written as a product of two non-constant polynomials.
It will be useful to write $P$ in the form
\begin{equation}P=P_d+P_{d-1} +\ldots + P_1+P_0 ,
 \label{degreepoly}
\end{equation}
where each $P_j$ is a \emph{homogeneous} polynomial of degree $j$ for $0\leq j\leq d$. We may then identify $P(x,y)$ with a homogeneous polynomial in three variables  
defined by
\[\tilde{P}(x,y,z) = \sum c_{ij}x^iy^jz^{d-i-j}=P_d(x,y) +zP_{d-1}(x,y)+\ldots z^{d-1}P_1(x,y) +z^d P_0 .\]
Observe that $\tilde{P}$ is irreducible if and only if $P$ is.  
Since $\tilde{P}$ is homogeneous, given an algebraic curve $\Cu_P\subseteq\bb{C}^2$, we can define the projective compactification
\[\tilde{\Cu}_{\tilde{P}}=\{[x,y,z]\in\bb{CP}^2: \tilde{P}(x,y,z)=0\} \]
of $\Cu_P$, which is an algebraic curve in $\bb{CP}^2$. We define its \emph{points at infinity} to be
\[\tilde{\Cu}_{\tilde{P}}^\infty :=\tilde{\Cu}_{\tilde{P}}\cap\{[x,y,z]\in \bb{CP}^2: z=0\} .\]
Bezout's theorem 
implies that if $\tilde{P}$ is irreducible then $\tilde{\Cu}^\infty_{\tilde{P}}$ is a finite set (unless $\tilde{P}$ is a multiple of $z$).

We let $\Sing(\tilde{\Cu}_{\tilde{P}})$ denote the singular points of an irreducible projective curve $\tilde{\Cu}_{\tilde{P}}$; i.e.~the points $[x,y,z]$ on $\tilde{\Cu}_{\tilde{P}}$ such that 
\begin{equation*}\pard{\tilde{P}}{x}\Big|_{(x,y,z)}=\pard{\tilde{P}}{y}\Big|_{(x,y,z)}=\pard{\tilde{P}}{z}\Big|_{(x,y,z)}=0.
 \end{equation*}
Similarly, we let $\Sing(\Cu_P)$ denote the singular points of an irreducible algebraic curve $\Cu_P$. Even when $\Sing(\mathcal{C}_P)=\emptyset$, $\Sing(\tilde{\Cu}_{\tilde{P}})$ may be non-empty.  We therefore say that $\Cu_P$ has  \emph{no singularity at infinity} if 
$\Sing(\tilde{\Cu}_{\tilde{P}})\cap \tilde{\Cu}_{\tilde{P}}^{\infty}=\emptyset$.

\subsection{Algebraic properties of special Lagrangian surfaces}

We first make the following key observation.

\begin{prop}\label{complextoalgebraic}
 Suppose that $L$ is a 
 special Lagrangian with bounded area ratios in $\C^2$. Then after a hyperk\"ahler rotation, $L$ is  an algebraic curve.
 \end{prop}
\begin{proof}
 Lemma \ref{lem:SL.cplx} states that after a hyperk\"ahler rotation, $L$ is a holomorphic curve $C$.  Since we have assumed that $L$ satisfies \eqref{eq:bounded.area.ratios}, we may now apply \cite[Theorem 3]{Stoll} (in that paper $V(r,M)=\int_{M\cap B_r(0)} d\Ha^2$), which implies that $C$ is algebraic.
\end{proof}
We say $L$ is represented by an algebraic curve $\Cu_P$ if $\Cu_P$ and $L$ are equal as subsets of $\C^2$.
 Recall the definition of connected components in Remark \ref{rmk:components.cone}.
\begin{lem}\label{irreducible}
 Let $L$ be a connected special Lagrangian with bounded area ratios in $\bb{C}^2$. Then $L$ may be represented by  $\Cu_P$ where $P$ is irreducible.
\end{lem}
\begin{proof}
 Suppose, for a contradiction, that $L$ is represented by $\Cu_P$ where $P=\prod_{j=1}^lQ_j$ for $Q_j$ irreducible and degree at least one for $j\in \{1, \ldots,l\}$, and suppose, without loss of generality, that there is no representation with $P$ of lower degree. By Bezout's theorem either there exists $i,j\in \{1, \ldots,l\}$ such that $\Cu_{Q_i}\subset \Cu_{Q_j}$ or for all $i, j \in \{1, \ldots,l\}$ with $i\neq j$,  $\Cu_{Q_i} \cap \Cu_{Q_j}=E_{ij}$ where $E_{ij}$ is a finite set of points. 
 
 In the first case, $L$ is represented by $\Cu_{\widehat{P}}$ where $\widehat{P}=\frac{P}{Q_i}$. As $\widehat{P}$ has lower degree, this is a contradiction.
 
 In the second case, let $E$ be the union of the $E_{ij}$.  For all $j\in \{1, \ldots,l\}$, $\Cu_{Q_j}$ is closed in $\Cu_P$, hence $\iota^{-1}(\Cu_{Q_j})$ is closed in $L$, where $\iota$ is the immersion of $L$ in $\C^2$. Since $\iota$ is an immmersion and $E
 $ is a discrete set of points,  $L\setminus\iota^{-1}(E)$ is path connected. As $L=\cup_{j=1}^l\iota^{-1}(\Cu_{Q_j})$, we see that $L\setminus\iota^{-1}(E)$ is the disjoint union of the open sets $\iota^{-1}(\Cu_{Q_j}\setminus E)$. This contradicts connectedness if $l>1$.
 \end{proof}

\noindent By Lemma \ref{irreducible}, we may now restrict to irreducible polynomials $P$.

By \eqref{degreepoly}, for any $P$ of degree $d$, we may write $P=P_d+Q$  where $Q$ is of at most degree $d-1$.  Clearly if both $|x|$ and $|y|$ are large, $P_d$ will dominate $Q$. We will show that $P_d$ corresponds to the blow-down of $\mathcal{C}_P$.

For $P$ in \eqref{eq:P.poly} and $\l>0$ we consider the scalings $\l\Cu_P$.  Notice  $\l\Cu_P=\Cu_{P^{\lambda}}$,
where 
\begin{equation}\label{eq:P.lambda}
P^\l(x,y) =\l^dP\left(\frac{x}{\l}, \frac{y}{\l}\right) =\l^d \sum_{i,j} \frac{c_{ij}}{\l^{i+j}}x^iy^j.
\end{equation}
 An easy lemma \cite[Lemma 2.8]{Kirwan} implies there exists a positive integer $D\leq d$ and, for $j=1,\ldots,D$, integers $1\leq m_j\leq d$ with $\sum_{j=1}^Dm_j=d$ and $\alpha_j, \beta_j\in\bb{C}$ such that
\begin{equation}P_d(x,y) = \prod_{j=1}^D(\a_jx+\b_j y)^{m_j},
 \label{FactorP0}
\end{equation}
where there does not exist $\mu\in \bb{C}$ such that $\mu(\a_j, \b_j) = (\a_k,\b_k)$ for $j\neq k$. Hence, $\mathcal{C}_{P_d}$ is a union of $D$ transverse planes with multiplicities $m_j$ for $j=1,\ldots,D$ and
\begin{equation}P^\l = P_d +\l Q_\l = \prod_{j=1}^D(\a_jx+\b_j y)^{m_j} +\l Q_\l ,
 \label{blowdownpolyform}
\end{equation}
where $Q_\l$ is a polynomial with coefficients which are bounded as $\l\ra0$.

We now have the following, where we observe that we may associate a Radon measure $\mu_{P}$ to $\Cu_{P}$: for any Borel set $S\subset\bb{C}^2$,
\[\mu_{P}(S) = \int_{\Cu_{P}\cap S} m_x({P}) d\mathcal{H}^2(x),\]
where $m_x({P})$ is the multiplicity of the zero of ${P}$ at $x$. 
\begin{prop}\label{P0limit}
 Suppose that $P$ is irreducible and written as $P=P_d+Q$ where $P_d$ is homogeneous of degree $d$ and $Q$ is of degree at most $d-1$. Then $\lambda\Cu_P \ra \Cu_{P_d}$ in the sense of Radon measures as $\l\ra 0$; i.e.~the blow-down of $\Cu_P$ is $\Cu_{P_d}$.  Moreover, locally away from the origin, $\lambda\Cu_P$ can be written as a multivalued graph over $\Cu_{P_d}$ which converges smoothly to $0$ as $\lambda\ra 0$.
 
 Hence, writing $P$ as in \eqref{blowdownpolyform}, the blow-down of $\Cu_P$ consists of $D$ transverse planes with multiplicities $m_j$ for $j=1,\ldots,D$, which determines $P_d$ up to scale.
\end{prop}

\begin{proof}
We proceed in three steps.

\smallskip

\noindent{\bf Step 1}: \emph{The limit Radon measure is supported on $\Cu_{P_d}$.} 

\noindent Let $B\subseteq\bb{C}^2$ be any Borel set with compact closure such that on $\ov{B}$, $|P_d|>0$. By continuity there are positive $\delta, C$ depending on $B$ and $P$ such that on $B$, $|P_d|>\d$ and $|Q_\l|<C$ for all $0\leq\l<1$.  
Therefore, there is $\l_0(B)>0$ sufficiently small such that for all $\l<\l_0$, $|P^\l|>\frac{\d}{2}$ on $B$. The claim now follows.

\smallskip

\noindent{\bf Step 2}: \emph{Outside a finite set $E$, $\mathcal{C}_P$ may be locally written as a holomorphic graph over the planes in $\mathcal{C}_{P_d}$.}

\noindent For $P^\l$ as in \eqref{blowdownpolyform}, for each $j\in\{1,\ldots,D\}$ we consider a change of complex coordinates $(\hat{x},\hat{y})$ in $\text{SU}(2)$ such that $\hat{y}=\mu(\a_jx+\b_jy)$ for some $\mu \in \bb{R}$, and write  $P^{\l,j}$ for $P^{\l}$ in these coordinates. We see that $\pard{P^{1,j}}{\hat y}$ consists of $c_j\leq d-1$ irreducible factors. 
 Bezout's theorem implies there is a finite set of points $E^j$ such that $P^{1,j}$ and $\pard{P^{1,j}}{\hat y}$ are zero. We then set $E=\cup_{j=1}^DE^j$. The holomorphic implicit function theorem (see e.g.~\cite[Theorem B.1]{Kirwan}) now implies the claim.
\smallskip

For some $j$, using coordinates $(\hat{x},\hat{y})$ as above, we  write $\bb{C}^2=\bb{C}\times\bb{C}=\Cu_{\hat{y}}\times\Cu_{\hat{x}}$. 
We choose a simply connected open set $B_1\subset\subset \Cu_{\hat{y}}$ such that $\operatorname{dist}(B_1,0)=\delta_1>0$, and set $B_2 = \{\hat{y} \in \Cu_{\hat{x}} : |\hat{y}|<\delta_2\}$ for some $\delta_2>0$. We define $B=B_1\times B_2\subseteq\bb{C}\times\bb{C}$, and assume that $\delta_2$ is sufficiently small so that $B$ does not intersect any of the other planes in $\Cu_{P_d}$. Since $E$ is a finite set and $0\notin B$, we may choose $\l<\l_1(\delta_1,E)$ small enough such that $\l E \cap B=\emptyset$. We observe that in $B$, $\pard{P^{\l,j}}{\hat{y}}\neq 0$. 

\smallskip

\noindent \textbf{Step 3:} \emph{For $\l<\l_2(\delta_2,P, \l_1)$, $\l\Cu_P\cap B$ may be written as a graph over the planes in $\mathcal{C}_{P_d}$, and the graph functions converge smoothly to zero as $\l\ra 0$.}  

\noindent Fixing $\hat{x}\in B_1$, each $(\hat{x},0)\in B$ is a root of $P_d$ of order $m_j$. By continuity of the roots of polynomials with respect to change of coefficients, for $\l<\l_2(\delta_2,P)<\l_1$ small enough there are $m_j$ roots, counting multiplicity, of $P^{\l,j}(\hat{x},\hat{y})$ in $B$. These roots are distinct and multiplicity one, since $\pard{P^{\l,j}}{\hat{y}}\neq 0$. The holomorphic implicit function theorem now implies that for $\l<\l_2$, $\Cu_{P^\l}\cap B$ is made up of $m_j$ disjoint holomorphic graphs over $B_1\subset\Cu_{\hat{y}}$. As $B\cap\{(\hat x,\hat y)\in\bb{C}^2: |\hat y|>\e\}$ is not contained in the support of $\Cu_{P_d}$, Step 1 implies that these graph functions converge uniformly to $\Cu_{\hat{y}}\cap B$ as $\l\ra 0$. Standard properties of holomorphic functions imply that this convergence is smooth. Therefore, the area integrals converge on this region to $m_j \Ha^2(B\cap\Cu_{\hat{y}})$ and the claimed convergence follows.

\smallskip

 Any Borel set $S\subseteq\bb{C}^2\setminus\{0\}$ (up to a set of measure zero) is a countable union of Borel sets $B$ as above, and a set outside the support of $\Cu_{P_d}$, so the result follows.
\end{proof}

A similar proof to Propostion \ref{P0limit} allow us to describe the blow-up of $\Cu_P$. 
\begin{prop}\label{Blowup}
 Suppose $0\in \Cu_P$ and write $P$ as in \eqref{degreepoly}. Suppose that $k$ is the least such that $P_k\neq 0$. Then, as Radon measures, $\lambda\Cu_P \ra \Cu_{P_k}$
 as $\l\ra\infty$. 
\end{prop}

We now observe that we can detect when $L$ is embedded using $\Sing(\mathcal{\Cu}_P)$.

\begin{lem}\label{embeddedsingular}
 A connected special Lagrangian $L$ with bounded area ratios in $\C^2$ is embedded if and only if it may be represented by $\Cu_P$ with $\Sing(\Cu_P)=\emptyset$.
\end{lem}
\begin{proof}
By Lemma \ref{irreducible}, we may assume that $P$ is irreducible.  
If $\Sing(\Cu_P)=\emptyset$, the complex implicit function theorem \cite[Theorem B.1]{Kirwan} implies $L$ is embedded.
 
 Suppose that $L$ is embedded and represented by $\Cu_P$. Suppose for a contradiction that (without loss of generality) $0\in \Sing(\Cu_P)$. 
 Writing $P$ as in \eqref{degreepoly}, we deduce that $P_0=P_1=0$.
 Proposition \ref{Blowup} implies that blowing up $\Cu_P$ around $0$ converges to $\Cu_{P_k}$ where $k\geq 2$. Applying \cite[Lemma 2.8]{Kirwan} implies that 
 $\Cu_{P_k}$ contains 
 $k$ planes (counted with mulitplicity).  However,  $\Cu_P$ is embedded, so the blow-up around any point is a single multiplicity one plane, which gives the required contradiction.
\end{proof}

\begin{prop} \label{Nondeginfinity}
Let $L$ be a connected special Lagrangian with bounded area ratios in $\C^2$. Suppose that a blow-down of $L$ consists of $d$ planes counted with multiplicity, and $D$ planes counted without multiplicity, and that $L$ has no singularities at infinity. Then
\begin{equation}\label{eq:nondeginfinity.totcurv}
\int_L |A|^2 d\Ha^2 \leq 4\pi\left(d(d-2)+D\right),
\end{equation}
where equality holds if and only if $L$ is embedded.  
\end{prop}
\begin{proof}   By Lemma \ref{irreducible}, we have $\tilde{\Cu}_{\tilde{P}}\subseteq\bb{CP}^2$ defined by an irreducible polynomial of degree $d$, which is the projective compactification of $\Cu_P$ representing $L$.  We shall use Gauss--Bonnet and the minimality of $L$ to obtain our total curvature bound, since we know that the Gauss curvature $K$ of $L$ satisfies
\begin{equation}\label{eq:K.A}
2K=|H|^2-|A|^2=-|A|^2.
\end{equation}

By \cite[Theorem 7.12]{Kirwan}, we have a compact Riemann surface $M$ 
and a biholomorphic mapping $\pi:M \setminus \pi^{-1}(\operatorname{Sing}(\tilde{\Cu}_{\tilde{P}}))\ra \tilde{\Cu}_{\tilde{P}} \setminus \operatorname{Sing}(\tilde{\Cu}_{\tilde{P}})$. Noether's theorem \cite[Corollary 7.34, Theorem 7.37]{Kirwan} implies that the Euler characteristic 
 of $M$ is 
\begin{flalign*}\chi(M) &= d(3-d) + 2 \sum_{p\in\text{Sing}(\tilde{\Cu}_{\tilde{P}})} \delta(p),
\end{flalign*}
where $\delta(p)$ is a positive integer associated to singular points as defined in \cite[Definition 7.35, Remark 7.36]{Kirwan}. We deduce that
\begin{equation}\label{eq:chi.L}
\chi(L) = d(3-d) + 2 \sum_{p\in\text{Sing}(\tilde{\Cu}_{\tilde{P}})} \delta(p) - d^{\infty}
\end{equation}
 where $d^{\infty} = \#(\pi^{-1}(\tilde{\Cu}^\infty_{\tilde{P}}))$.  
 
Gauss--Bonnet implies that for any open set $U\subset L$,
\begin{equation}\label{eq:local.GB}
\int_U Kd\Ha^2 + \int_{\partial U} \kappa_g d\Ha^1 =2\pi \chi(U),
\end{equation}
where $\kappa_g$ is the geodesic curvature of the boundary.  Recall the rescalings $P^{\lambda}$ of $P$ in \eqref{eq:P.lambda} for $\lambda>0$.  Notice that the integrals in \eqref{eq:local.GB} are scale-invariant.

Consider a positive sequence $\l_j \ra 0$ as $j\ra \infty$ and let $U_j=B_1(0)\cap \Cu_{P^{\l_j}}$.  Proposition \ref{P0limit} implies that, for $j$ sufficiently large, $\Cu_{P^{\l_j}}$ may be written locally graphically over the planes in the blow-down $\Cu_{P_d}$, away from the origin. 
 For $j$ large enough, the boundary $\partial U_j$ may be divided into $D$ components (not necessarily connected) according to which plane in $\Cu_{P_d}$ the graph defining $U_j$ locally converges. Choose one of the planes in $\Cu_{P_d}$ and rotate coordinates so that this plane is $\Cu_y$. Suppose that the multiplicity of $\Cu_y$ in $\Cu_{P_d}$ is $m$ and denote the component of $\partial U_j$ converging to $\Cu_y$ by $\tilde \gamma_j$.  Then, due to the local convergence of $U_j$, $\tilde \gamma_j$ will converge locally smoothly to $m$ copies of $\tilde{\gamma}(t)=(e^{it}, 0)$. We therefore have that
\[\lim_{j\ra\infty}\int_{\tilde{\gamma}_j} \kappa_g d\Ha^1 = 2\pi m .\]
Repeating the above for each plane in $\Cu_{P_d}$, using \eqref{eq:local.GB} and scale-invariance gives 
\begin{equation}\label{eq:L.GB.2}
\int_L Kd\Ha^2 = \lim_{j\ra\infty}\int_{U_j} Kd\Ha^2 =  2\pi (\chi(L)-d).
\end{equation}

Combining \eqref{eq:K.A}, \eqref{eq:chi.L} and \eqref{eq:L.GB.2}  leads to
\begin{equation}\int_L |A|^2 d\Ha^2 = 4\pi\left(d(d-2)-2 \sum_{p\in\text{Sing}(\tilde{\Cu}_{\tilde{P}})} \delta(p) +d^{\infty} \right).
 \label{generaltotalcurv}
\end{equation}

Since $L$ has no singularities at infinity, each point at infinity of $\tilde{\Cu}_{\tilde{P}}$ corresponds to exactly one point in $M$, and so $d^{\infty}=D$. The claimed inequality \eqref{eq:nondeginfinity.totcurv} now holds. If we have equality in \eqref{eq:nondeginfinity.totcurv}, then 
$\sum_{p\in\text{Sing}(\tilde{\Cu})} \delta(p)=0$ by \eqref{generaltotalcurv}. 
 As $\delta(p)\geq1$, this implies there are no singular points and so Lemma \ref{embeddedsingular} yields the proposition.
\end{proof}

We can now put our results together to give the following.

\begin{theorem}\label{SLsurface}
 Suppose that $L$ is a connected special Lagrangian with bounded area ratios in $\C^2$ with a blow-down consisting of $d$ planes counted with multiplicity. Then, after a hyperk{\"a}hler rotation, $L$ may be written as the zero set of a complex polynomial of degree $d$.  Furthermore,
\[\int_L |A|^2 d\Ha^2 \leq 4\pi d(d-1),\]
where equality holds if and only if $L$ is embedded and its blow-down is $d$ distinct multiplicity one planes.  \end{theorem}
\begin{proof}
 The first part follows from Proposition \ref{complextoalgebraic}, Lemma \ref{irreducible}, Proposition \ref{P0limit} and Lemma \ref{embeddedsingular}. The rest of the result follows almost exactly as in Proposition \ref{Nondeginfinity}. The difference is that $d^{\infty}$ given in the proof of that proposition is not precisely determined: instead we claim that $d^{\infty}\leq d$. To see this, first notice that the proof of Proposition \ref{P0limit} implies there exists an $R>0$ such that there can be at most $d$ components to $L\setminus B_R$.  Recall the map $\pi$ in the proof of Proposition \ref{Nondeginfinity}. Observe that each sufficiently small neighbourhood $U_p$ of $p\in\pi^{-1}(\tilde{\Cu}_{\tilde{P}}^\infty))$ is such that $\pi:U_p\setminus\{p\}\to\pi(U_p\setminus\{p\})$ is a bijection whose image is contained in a component of $L\setminus B_R$. If $d^{\infty}=\#( \pi^{-1}(\tilde{\Cu}_{\tilde{P}}^\infty))\geq d$, this would be a contradiction to the fact that $\pi$ is biholomorphic away from singular points.
 The proposition now follows as in the proof of Proposition \ref{Nondeginfinity}.
\end{proof}

\subsection{Degree 2}

We now classify special Lagrangians satisfying \eqref{eq:bounded.area.ratios} in $\C^2$ whose blow-down is two planes $P_1$, $P_2$, counted with multiplicity (so $P_1$ can equal $P_2$).

\begin{theorem}\label{degreelessthantwo}
Suppose that $L$ is a special Lagrangian with bounded area ratios in $\C^2$ with a blow-down consisting of two planes counted with multiplicity. Then, up to rigid motions, either
\begin{itemize}
\item[(a)] $L$ is two planes, or
\item[(b)] $L=L_{\phi}(a+ib)$ given in Example \ref{ex.SL.dim2} and $\int_L|A|^2 d\Ha^2=8\pi$, or
\item[(c)] $L=L(a+ib)$ given in Example \ref{ex:SL.z2} and $\int_L|A|^2 d\Ha^2=4\pi$.
\end{itemize}
\end{theorem}
\begin{proof}
We first suppose that $L$ is connected. Proposition \ref{complextoalgebraic}, Lemma \ref{irreducible}, equation \eqref{blowdownpolyform} and Proposition \ref{P0limit} imply that $L$ may be represented by an algebraic curve $\Cu_P$ where $P$ is irreducible and may be written (where all coefficients are in $\C$): \[P(x,y) = (\a_1 x+\b_1 y)(\a_2 x+\b_2y) +cx+dy+f.\]
Applying an $\text{SU}(2)$ transformation, we can rewrite $P$  as
\[P(x,y) = x(ax+by) +cx+dy+e.\]
Suppose that $L$ is not equal to the blow-down. We then split into two cases.
\subsubsection*{Case 1: $b\neq 0$}
We may complete squares to yield
\begin{flalign*}
P(x,y)
&=(x+db^{-1})(ax+by+(c-db^{-1}a)) +e-db^{-1}(c-db^{-1}a) .
\end{flalign*}
We see that, after hyperk\"ahler rotation, this is a translation of one of the special Lagrangians in Example \ref{ex.SL.dim2}. There are no singularities at infinity and two distinct planes in the blow-down, so Proposition \ref{Nondeginfinity} gives the formula for the total curvature.

\subsubsection*{Case 2: $b=0$}
As $P$ is irreducible and $L$ is non-planar we must have $d\neq 0$ and
\[y=-d^{-1}(ax^2+cx+e)=-d^{-1}(a(x+\frac{c}{2a})^2+e-\frac{c^2}{4a^2}.\]
A short calculation shows that there is no singular point at infinity, but there is only one plane in the blow-down (counted without multiplicity). Proposition \ref{Nondeginfinity} then gives the total curvature of $L$.\medskip

If $L$ is not connected then, since it has bounded area ratios, it has a finite number of connected components $L_1,\ldots,L_k$. Applying Proposition \ref{complextoalgebraic} and Lemma \ref{irreducible} to each $L_j$, we get irreducible polynomials $P^j$ of degree $d_j\geq 1$ so that $L_j=\Cu_{P^j}$. Writing $P^j = P^j_{d_j}+Q^j$ as in \eqref{blowdownpolyform}, Proposition \ref{P0limit} implies that the blow-down of $L_j$ is  determined by $P^j_{d_j}$. Since $L =\bigsqcup_{j=1}^k L_j
$ has a blow-down consisting of two planes (counted with multiplicity), we must have $\sum_{j=1}^k d_j =2$. As $L$ is not connected, $L$ must consist of two components of degree one, so $L$ is a disjoint union of two parallel planes determined by its blow-down.
\end{proof}

\section{Applications}\label{sect:applications}

\subsection{Topology of blow-ups}

We begin with the following simple but important observation which allows us to utilize our structure theory, which is given in \cite{Cooper}.  

\begin{prop}\label{prop:top.blowup}
Let $(L_t)_{0\leq t<T}$ be a solution to LMCF with uniformly bounded area ratios which has a singularity at time $T$. Let $L^{\infty}_s$ be a Type II blow-up of $L_t$ as given in Definition \ref{defn:type.II}. 
\begin{itemize}
\item[(a)] If the singularity is Type I, i.e.~\eqref{eq:Type.I.blowup} holds, then $L^{\infty}_s$ is a self-shrinker and hence is monotone.
\item[(b)] Otherwise, the singularity is Type II and $L^{\infty}_s$ is exact and zero-Maslov. 
\end{itemize}
\end{prop}

\begin{proof}
For (a), the fact that the blow-up is a self-shrinker is a standard consequence of the monotonicity formula \eqref{eq:monotonicity}.  Part (a) then follows by Lemma \ref{lem:self.monotone}.

For (b), we may view $L_t$ as the image of a family of immersions $\iota_t:L\to L_t$ for some fixed manifold $L$.  We write $L_s^i$ for the Type II rescalings as in Definition \ref{defn:type.II}, and observe that the curvature bound along with standard estimates and a diagonal argument imply $C^\infty_{\text{loc}}$ convergence to $L^\infty_s$ (cf.~\cite[Proposition 1.1]{Smoczyk.Longtime}). 

Suppose that $H^1(L^{\infty}_s)\neq 0$, otherwise the result is trivial.  Then there is a smooth closed curve $\Gamma\subset L^\infty_s$ with $[\Gamma]\in H_1(L^\infty_s)$ such that $[\Gamma]\neq 0$. We have a sequence of smooth curves $\Gamma_i\subset L^i_s$ such that the $\Gamma_i$ converge to $\Gamma$ smoothly. We consider the sequence of curves $\gamma_i\subset L$ such that $\Gamma_i = \sigma_i(\iota_{t_i+\sigma_i^{-2}s}\circ \gamma_i-x_i)$.
We have that
\begin{align}\label{eq:int.gamma.1}
\int_{\Gamma}\lambda=\lim_{i\ra\infty}\int_{\Gamma_i}\lambda  = \lim_{i\to\infty}
\int_{\gamma_i}\sigma_i^2(\iota_{t_i+\sigma_i^{-2}s})^*\lambda,
\end{align}
where the extra factor of $\sigma_i$ arises from the fact that we scale both $\lambda$ and $L_t$.  
Lemma \ref{lem:evol.eqs} implies that $[\rd\theta_t]=[\rd\theta_0]$ and 
\begin{equation}\label{eq:lambda.cohom.1}
[(\iota_{t_i+\sigma_i^{-2}s})^*\lambda]-[\iota_0^*\lambda]=
-2(t_i+\sigma_i^{-2}s)[\rd\theta_0].
\end{equation}
Hence, combining \eqref{eq:int.gamma.1} and \eqref{eq:lambda.cohom.1}, we have
\begin{align}
\int_{\Gamma}\lambda
&=\lim_{i\to\infty}\sigma_i^2\int_{\gamma_i}\iota_0^*\lambda-2(t_i+\sigma_i^{-2}s)\rd\theta_0\nonumber\\
&=-2s\lim_{i\ra\infty}\int_{\gamma_i}\rd\theta_0+\lim_{i\to\infty}\sigma_i^2\int_{\gamma_i}\iota_0^*\lambda-2t_i\rd\theta_0.\label{eq:int.gamma.2}
\end{align}
The left-hand side of \eqref{eq:int.gamma.2} is finite, and $\gamma_i$ is a representative of a finite homology class, and so since $\sigma_i\to\infty$ and $t_i\to T$ by assumption, we must have 
\begin{equation}\label{eq:lambda.cohom.2}
\lim_{i\ra\infty}\int_{\gamma_i}\iota_0^*\lambda-2T\rd\theta_0=0.
\end{equation}
Re-inserting this in \eqref{eq:int.gamma.2}, we see that we need
\begin{equation*}
\lim_{i\to\infty}\sigma_i^2(T-t_i)\int_{\gamma_i}\rd\theta_0<\infty.
\end{equation*}
The Type II assumption therefore forces 
\begin{equation*}
\lim_{i\ra\infty}\int_{\gamma_i}\rd\theta_0=0,
\end{equation*}
but since the possible values of path integrals of $\rd \theta_0$ is a discrete set containing $0$, for all $i$ sufficiently large,
\begin{equation}\label{eq:theta.cohom.1}
 \int_{\gamma_i}\rd\theta_0=0 .
\end{equation}
By scale invariance of the Lagrangian angle and the preservation of its cohomology class along the flow we have 
\begin{equation*}
\int_{\Gamma_i}\rd\theta_s=0  
\end{equation*}
and, due to smooth convergence of $L^i_s$ to $L^\infty_s$,
\begin{equation}\label{eq:theta.cohom.2}
\int_{\Gamma}\rd\theta_s^\infty=0. 
\end{equation}

Similarly, \eqref{eq:lambda.cohom.2} implies that $\int_{\gamma^i} \iota_0^* \lambda =0$ for $i$ large, and so \eqref{eq:int.gamma.2} and \eqref{eq:theta.cohom.1} force
\begin{equation}\label{eq:lambda.cohom.3}
\int_{\Gamma}\lambda=0.
\end{equation}
Since \eqref{eq:theta.cohom.2} and \eqref{eq:lambda.cohom.3} hold for all such $\Gamma$,   $L^{\infty}_s$ is exact and zero-Maslov.
\end{proof}

Proposition \ref{prop:top.blowup} has the following interesting consequence. 

\begin{cor} If $b\neq 0$, the special Lagrangian $L_{\phi}(a+ib)$ in $\C^2$ in Example \ref{ex.SL.dim2} cannot arise as a Type II blow-up of a solution to LMCF as in Proposition \ref{prop:top.blowup}.
\end{cor}

\begin{proof}
 Since $L_{\phi}(a+ib)$ for $b\neq 0$ is zero-Maslov but not exact, it is also not monotone, and so cannot arise as a Type II blow-up by Proposition \ref{prop:top.blowup}. 
\end{proof}

\subsection{Lawlor necks as singularity models}
We now give our first application.

\begin{proof}[Proof of Theorem \ref{thm:Lawlor}]
If $P_1$ and $P_2$ have distinct Lagrangian angles, Proposition \ref{prop:planar.blow-down} implies we are in case (a). If $P_1$ and $P_2$ have the same Lagrangian angle, Proposition \ref{prop:SL.blowdown} implies that $L_t$ is special Lagrangian, and the result follows from Theorem \ref{thm:SL.planar.uniqueness}. 
\end{proof}

\begin{remark}
Theorem \ref{thm:Lawlor} is false if we do not assume the planes are transverse.  For example, the translators in Example \ref{ex:translators} are exact, almost calibrated, eternal solutions which are not special Lagrangian and whose blow-down is a pair of planes intersecting along a line.  
\end{remark}

\subsection{Harvey--Lawson \texorpdfstring{$T^2$}{T2}-cone} We now rule out the Harvey-Lawson $T^2$ as a blow-down.

\begin{proof}[Proof of Theorem \ref{thm:HL}]
Suppose for a contradiction that $C$ occurs as a blow-down. Since $C$ is special Lagrangian, Proposition \ref{prop:SL.blowdown} implies that the Type II blow-up is special Lagrangian.  Work by Imagi \cite{Imagi} then shows that a smooth special Lagrangian in $\C^3$ converging to $C$ as a varifold must be one of the Harvey--Lawson smoothings $L_j(a_j)$ given in Example \ref{ex:HL}.  However, the singularity must be Type II as $(L_t)_{0\leq t<T}$ is zero-Maslov, so Proposition \ref{prop:top.blowup} implies that the Type II blow-up is exact. This contradicts the fact that the Harvey--Lawson smoothings are not exact. 
\end{proof}

\subsection{Multiplicity two plane} Finally, we classify the ancient solutions whose blow-down is a multiplicity two plane.

\begin{proof}[Proof of Theorem \ref{thm:dimension2}]
 Since $(L_t)_{0\leq t<T}$ is almost calibrated and $P$ has a single Lagrangian angle, Proposition \ref{prop:SL.blowdown} implies that $(L_t)_{0\leq t<T}$ is special Lagrangian for all $t<0$. Theorem 
 \ref{degreelessthantwo} yields the claimed result.
\end{proof}

\providecommand{\bysame}{\leavevmode\hbox to3em{\hrulefill}\thinspace}
\providecommand{\MR}{\relax\ifhmode\unskip\space\fi MR }
\providecommand{\MRhref}[2]{%
  \href{http://www.ams.org/mathscinet-getitem?mr=#1}{#2}
}
\providecommand{\href}[2]{#2}

\end{document}